\documentclass[letter, 11 pt, reqno]{amsart}

\usepackage[hyphens]{url} 
\usepackage{amsthm}
\usepackage{graphicx}
\usepackage{bbm}
\usepackage{tasks}
\usepackage{fancyhdr}
\usepackage[T1]{fontenc}
\usepackage[utf8]{inputenc}
\usepackage{mathtools}
\usepackage{amsmath,amsfonts}
\usepackage[inline]{enumitem}
\usepackage{mathrsfs}
\usepackage{textcomp}
\usepackage{tikz-cd}
\usepackage{polski}
\usepackage{array}
\usepackage{amssymb}
\usepackage[new]{old-arrows}
\usepackage{kantlipsum}
\usepackage{hyperref}
\usepackage{cleveref}
\usepackage{moreenum}
\usepackage[english]{babel}
\usepackage[scr=boondox]{mathalfa}
\usepackage{aliascnt}
\usepackage{etoolbox}
\usepackage{stackengine}
\usepackage{esvect}
\usepackage{mleftright}
\addto\extrasenglish{%
}

\hypersetup{
    colorlinks=true,
    linkcolor=black,
    filecolor=black,      
    urlcolor=black,
    citecolor=black
}

\makeatletter
\patchcmd{\@maketitle}
  {\ifx\@empty\@dedicatory}
  {\ifx\@empty\@date \else {\vskip1ex \centering\footnotesize\@date\par\vskip1ex}\fi
   \ifx\@empty\@dedicatory}
  {}{}
\patchcmd{\@adminfootnotes}
  {\ifx\@empty\@date\else \@footnotetext{\@setdate}\fi}
  {}{}{}
\makeatother

\calclayout

\graphicspath{ {./images/} }

\theoremstyle{plain}
\newtheorem{thm}{Theorem}[section]

\newaliascnt{prop}{thm}
    \newtheorem{prop}[prop]{Proposition}
    \aliascntresetthe{prop}

\newaliascnt{lem}{thm}
    \newtheorem{lem}[lem]{Lemma}
    \aliascntresetthe{lem}

\newaliascnt{cor}{thm}
    \newtheorem{cor}[cor]{Corollary}
    \aliascntresetthe{cor}

\newaliascnt{conj}{thm}
    
    \aliascntresetthe{conj}

\theoremstyle{definition}

\newaliascnt{exm}{exmp}
    \newtheorem{exm}[exm]{Example}
   \aliascntresetthe{exm}

\newaliascnt{dfn}{thm}
    \newtheorem{dfn}[dfn]{Definition}
    \aliascntresetthe{dfn}

\newaliascnt{rem}{thm}
    \newtheorem{rem}[rem]{Remark}
    \aliascntresetthe{rem}

\numberwithin{equation}{section}

\newcommand{\R}{\mathbb{R}}
\newcommand{\C}{\mathbb{C}}
\newcommand{\Z}{\mathbb{Z}}

\newcommand{\M}{\mathcal{M}}
\newcommand{\g}{\mathfrak{g}}
\newcommand{\s}{\mathfrak{s}}

\newcommand{\su}{\mathfrak{su}}

\newcommand{\G}{\mathscr{G}}
\newcommand{\A}{\mathscr{A}}
\newcommand{\B}{\mathscr{B}}

\newcommand\omicron{o}

\DeclareMathOperator{\cok}{coker}
\DeclareMathOperator{\im}{im}
\DeclareMathOperator{\ch}{ch}
\DeclareMathOperator{\orb}{orb}

\DeclareMathOperator{\Rea}{Re}
\DeclareMathOperator{\Ima}{Im}

\DeclareMathOperator{\Hom}{Hom}
\DeclareMathOperator{\Map}{Map}

\DeclareMathOperator{\id}{id}

\DeclareMathOperator{\Sym}{Sym}

\DeclareMathOperator{\ind}{ind}
\DeclareMathOperator{\LC}{LC}

\DeclareMathOperator{\Ind}{Ind}

\DeclareMathOperator{\Aut}{Aut}

\DeclareMathOperator{\tr}{tr}

\DeclareMathOperator{\Ad}{Ad}
\DeclareMathOperator{\ad}{ad}
\DeclareMathOperator{\Sf}{Sf}

\DeclareMathOperator{\Hess}{Hess}
\DeclareMathOperator{\Hol}{Hol}

\DeclareMathOperator{\PD}{PD}
\DeclareMathOperator{\End}{End}

\DeclareMathOperator{\grad}{grad}

\DeclareMathOperator{\rk}{rank}
\DeclareMathOperator{\vol}{vol}

\DeclareMathOperator{\cs}{cs}
\DeclareMathOperator{\pr}{pr}

\DeclareMathOperator{\dR}{dR}
\DeclareMathOperator{\Stab}{Stab}

\mathchardef\mhyphen="2D

\NewDocumentCommand{\evalat}{sO{\big}mm}{%
  \IfBooleanTF{#1}
   {\mleft. #3 \mright|_{#4}}
   {#3#2|_{#4}}%
}

\newcolumntype{M}[1]{>{\centering\arraybackslash}m{#1}}

\begin{document}

\title[Flat Connections over $G2$-Orbifolds]{On Counting Flat Connections over $G_2$-Orbifolds}
\author{Langte Ma}
\date{\small \today}
\address{Simons Center for Geometry and Physics, 100 Nicolls Road, Stony Brook, NY 11790}
\email{lma@scgp.stonybrook.edu}

\begin{abstract}
We study the moduli space of $G_2$-instantons on (projectively) flat bundles over torsion-free $G_2$-orbifolds. We prove that the moduli space is compact and smooth at the irreducible locus after adding small and generic holonomy perturbations. Consequently, we define an integer-valued invariant that is invariant under $C^0$-deformation of torsion-free $G_2$-structures. We compute this invariant for some orbifolds that arise in Joyce's construction of compact $G_2$-manifold. 
\end{abstract}

\maketitle

\section{\large \bf Introduction}

\subsection{\em Background}\label{ss1.1} \hfill
 
\vspace{3mm}

The study of gauge theory over manifolds with special holonomy was initiated by Donaldons--Thomas \cite{DT98} as an extension of the gauge theory techniques in dimension $3$ and $4$. In dimension $7$, the holonomy group over Riemannian manifolds that enters the picture is $G_2$, i.e. the automorphism group of octonions as a subgroup of $SO(7)$. For such manifolds, Donaldson--Segal \cite{DS11} proposed a program of constructing invariants by relating the moduli space of instantons and calibrated submanifolds, which was further elaborated by Doan--Walpuski \cite{DW19}. Ideally, one expects to have a thorough understanding of the moduli space of $G_2$-instantons, then extracts an enumerative invariant combining counting irreducible points and tracing through the potential failure of compactness. Invariants of such kind should be useful in the study of deformation properties of $G_2$-structures. 

Although a general understanding of the moduli space of $G_2$-instantons is still elusive, progress has been made through the study of explicit examples. The class of examples that motivate this paper comes from Joyce's construction \cite{J96}. Roughly speaking, these compact $G_2$-manifolds are obtained from flat $G_2$-orbifolds by resolving the singularties using families of hyperkähler ALE spaces. In \cite{W13}, Walpuski constructed examples of non-trivial $G_2$-instantons on Joyce's $G_2$-manifolds by developing the gluing technique that grafts families of finite energy ASD instantons on ALE spaces to flat connections on $G_2$-orbifolds. Therefore analyzing instantons over the building blocks would lead to a complete picture of the moduli space of $G_2$-instantons over $G_2$-manifolds from Joyce's list. 

In this paper, we carry out the first step that studies the moduli space of flat connections over $G_2$-orbifolds. Since the deformation theory of flat connections in dimension $7$ is not Fredholm, we choose to consider $G_2$-instantons over flat bundles. Heuristically, one can think of the $G_2$-structure as a deformation to the flatness equation of connections. To get a regular moduli space, one needs to further perturb the $G_2$-instanton equation by holonomy perturbations. After establishing the compactness of the moduli space, one can then extract an enumerative invariant by an appropriate counting procedure. The second step that studies $G_2$-instantons over a family of ALE spaces will be carried out in a future paper by the author in collaboration with Galt \cite{GM23}. This approach for analyzing moduli spaces of instantons also appeared in Kronheimer's work \cite{K91} of counting instantons over $K3$ surfaces. Our work is partially inspired by his method. 

\subsection{\em Main Results}\label{ss1.2} \hfill
 
\vspace{3mm}

The notion `orbifold' in this paper primarily means effective orbifold in the classical sense used by Satake \cite{S56} and Thurston \cite{T97}. When a general perspective of `orbifold' shows up, we will address it explicitly. We also recall basics of effective orbifold in \autoref{s2}. 

We write $(X, \phi)$ for a typical compact smooth $G_2$-orbifold which means the holonomy group of the Levi--Civita connection on the Riemannian orbifold $X$ is contained in $G_2 \subset SO(7)$. As explained below in \autoref{s2}, the information of such a structure is encoded in a $3$-form $\phi \in \Omega^3(X; \R)$ satisfying 
\begin{equation}\label{e1.1}
d\phi = 0 \qquad d \star_{\phi} \phi = 0,
\end{equation}
where $\star_{\phi}$ means the Hodge star defined with respect to the metric induced from $\phi$. One further requires that at each point $x \in X$ there is an orbifold chart $\tilde{U}_x \to U_x$ centered at $x \in U_x \subset X$ so that $\phi|_x$, when lifted to $\tilde{U}_x$, is identified with the standard $3$-form $\phi_0$ on $\R^7$:
\begin{equation}\label{e1.2}
\phi_0 := dx_{123} + dx_{145} +dx_{167} + dx_{246} - dx_{257} - dx_{347} - dx_{356},
\end{equation} 
where $dx_{ijk} := dx_i \wedge dx_j \wedge dx_k$. The condition \eqref{e1.1} is usually referred to as the torsion-free property of the $G_2$-structure. 

We write $P^{\dagger}$ for a principal $U(r)$-orbibundle over $(X, \phi)$ and $P$ the adjoint bundle of $P^{\dagger}$ with structure group $PU(r)$. Ignoring the regularity issue, we write $\A$ for the space of $C^{\infty}$-connections on the adjoint bundle $P$ and $\G$ the gauge group consisting of $C^{\infty}$-automorphisms of $P^{\dagger}$ that fixes the determinant line $\det P^{\dagger}$. Alternatively, one can think of $\A$ as the space of unitary connections on $P^{\dagger}$ that induce a fixed connection on $\det P^{\dagger}$, and $\G$ consists of automorphisms that preserve this space of connections. For a unitary connection $A^{\dagger}$ on $P^{\dagger}$, we write $A$ for the induced connection on the adjoint bundle $P$. 

\begin{dfn}
A unitary connection $A^{\dagger} \in \A$ is called a projective $G_2$-instanton with respect to $\phi$ if 
\[
F_A \wedge \star_{\phi} \phi = 0,
\]
where $A$ is the induced connection of $A^{\dagger}$ on the adjoint bundle $P$. 
\end{dfn}

In \autoref{s4}, we shall introduce a Banach space $\mathscr{P}$ of holonomy perturbations. Each element $\pmb{\pi} \in \mathscr{P}$ gives rise to a $\G$-invariant functional $\sigma_{\pmb{\pi}}: \A \to \Omega^6(X, \g_P)$ that can be used to perturb the $G_2$-instanton equation. Here $\g_P$ denotes the Lie algebra bundle associated to $P$ by the adjoint action. We write 
\[
\M_{\phi, \pmb{\pi}}(X, P^{\dagger}):=\{ A \in \A: F_A \wedge \star_{\phi} \phi + \sigma_{\pmb{\pi}}(A) = 0 \}/\G
\]
for the moduli space of $\pmb{\pi}$-perturbed projective $G_2$-instantons on $P^{\dagger}$.  Let $Z(r) \subset U(r)$ be the center of $U(r)$, which is a cylic group of order $r$. We say a connection $A \in \A$ is irreducible if its isotropy group $\Stab(A) \subset \G$ is $Z(r)$. We prove the following result concerning the structure of the moduli space.

\begin{thm}\label{t1.1}
Let $P^{\dagger}$ be a principal $U(r)$-orbibudle over a torsion-free $G_2$-orbifold $(X, \phi)$ whose Chern classes satisfy 
\begin{equation}\label{e1.3}
(r-1)c^2_1(P^{\dagger}) - 2rc_2(P^{\dagger}) = 0 \in H^4_{\orb}(X;\Z).
\end{equation}
Suppose $\pmb{\pi} \in \mathscr{P}$ is a generic small perturbation. Then the following hold.
\begin{enumerate}
\item The moduli space $\M_{\phi, \pmb{\pi}}(X, P^{\dagger})$ is compact. 
\item The irreducible locus $\M^*_{\phi, \pmb{\pi}}(X, P^{\dagger})$ is a smooth $0$-dimensional manifold, i.e. it's transversely cut out by the perturbed $G_2$-instanton equation. 
\end{enumerate}
\end{thm}

We explain why the torsion-free assumption \eqref{e1.1} plays an important role in \autoref{t1.1}. The closedness condition $d\phi = 0$ is used to provide an a priori bound on the $L^2$-norm of curvature for $G_2$-instantons. When the instanton equation is not perturbed, the requirement \eqref{e1.3} combined with the closedness of $\phi$ implies that all $G_2$-instantons are flat connections. The coclosed condition for $\phi$ is used essentially to get an elliptic deformation theory for $G_2$-instantons. 

We only sort out the orientability issue for the moduli space $\M_{\phi}(X, P^{\dagger})$ in limited cases of orbifolds. Concretely, we require the singular set of $X$ be either locally modeled on the fixed point set of a $\phi$-preserving involution or an intersection of the former models in certain proper sense. Locally, it forces the singular set to be either `associative' (calibrated submanifold) or a curve arising as the intersection of two associative submanifolds. Nevertheless, such orbifolds have covered most cases appeared in Joyce's construction. We call such orbifolds simple $G_2$-orbifolds. 

Besides \eqref{e1.3}, we also make another topological assumption on $P^{\dagger}$ for the purpose of avoiding reducible instantons. We call such bundles admissible in \autoref{d3.10}. With these preparations, we can define an enumerative invariant.

\begin{thm}
Let $P^{\dagger}$ be an admissible $U(r)$-bundle over a torsion-free, compact, smooth $G_2$-orbifold $(X, \phi)$. Then counting irreducible projective $G_2$-instantons on $P^{\dagger}$ defines an invariant 
\[
n_{\phi}(X, P^{\dagger}) := \#\M^*_{\phi, \pmb{\pi}}(X, P^{\dagger}) \in \Z/2. 
\] 
If we further require $(X, \phi)$ be a simple $G_2$-orbifold, and $c_1(P^{\dagger})$ even or $r=2$, then counting with signs gives us 
\[
n_{\phi}(X, P^{\dagger}) := \#\M^*_{\phi, \pmb{\pi}}(X, P^{\dagger}) \in \Z
\]
after fixing extra orientation data. Moreover, the invariant $n_{\phi}(X, P^{\dagger})$ is invariant under $C^0$-deformation of torsion-free $G_2$-structures on $X$. 
\end{thm}

Although we haven't found examples of torsion-free $G_2$-structures on a fixed smooth orbifold $X$ resulting in different values of $n_{\phi}(X, P^{\dagger})$, the torsion-free condition has been used in an essential way to deduce the deformation invariance. Because we need the closedness to guarantee compactness in \autoref{t1.1}. Note that the $h$-principle has been worked out for coclosed $G_2$-structures by Crowley--Nordström \cite{CN15}. So connecting torsion-free $G_2$-structures by a path of coclosed $G_2$-structures is not hard to achieve. It would be interesting to understand how the moduli space of admissible bundles can change at a jumping point where the $G_2$-structure starts to fail being closed. 

Since we work on (projectively) flat bundles, there is chance that the invariant we defined is merely topological. This is indeed the case if the moduli space of projectively flat connections is already non-degenerate in the following sense. 

\begin{dfn}
Let $A^{\dagger}$ be an irreducible projectively flat connection on an admissible principal $U(r)$-bundle $P^{\dagger}$. We say $A$ is non-degenerate if 
\[
H^1(X, \rho_A) = 0,
\]
where $\rho_A: \pi^{\orb}_1(X) \to SO(r^2-1)$ is the representation associated to the induced flat connection on the bundle $\su(P^{\dagger})$ consisting of traceless skew-Hermitian endomorphisms of $P^{\dagger}$. 
\end{dfn}

\begin{prop}\label{p1.5}
Let $P^{\dagger}$ be an admissible $U(r)$-bundle over a torsion-free, compact, smooth $G_2$-orbifold $(X, \phi)$. Suppose every irreducible projectively flat connection on $P^{\dagger}$ is non-degenerate. Then the invariant $n_{\phi}(X, P^{\dagger})$ is independent of the choices of the torsion-free $G_2$-structure $\phi$. 
\end{prop}

We ask for the torsion-free condition in \autoref{p1.5} to ensure that the invariant $n_{\phi}(X, P^{\dagger})$ is defined. In the non-degenerate case, the moduli spaces together with their deformation structures are identified canonically with respect to different torsion-free $G_2$-structures. When $(X, \phi)$ is a flat orbifold, a Weitzenböck formula argument shows that every irreducible projectively flat connection is non-degenerate automatically. In particular, all orbifolds arising from Joyce's construction satisfy the non-degeneracy condition automatically. We compute the invariant for one such orbifold by exhausting all irreducible representations. Such computation for many other flat orbifolds in Joyce's construction can be worked out in a similar way.

\begin{prop}\label{p1.6}
Let $(X, \phi)$ be the flat simple orbifold constructed in Example 3 of \cite{J96}. Then there are admissible $U(2)$-bundles $P^{\dagger}$ on $(X, \phi)$ whose corresponding invariant is $n_{\phi}(X, P^{\dagger}) = 2^6 \in \Z$. 
\end{prop}

To conclude, we mention that a homotopy invariant $\nu(X, \phi) \in \Z/48$ for $G_2$-structures (not necessarily torsion-free) over $7$-manifolds was introduced by Crowley--Nordström \cite{CN15} as an `$\hat{A}$-defect'. If we consider non-simply-connected $G_2$-manifolds $(X, \phi)$ with $H^1(X; \Z) = 0$, an invariant $n_{\phi}(X) \in \Z$ can be defined by summing up the counting of all irreducible $G_2$-instantons over flat $SU(2)$-bundles. We suspect there might be a relation between $n_{\phi}(X)$ and $\nu(X, \phi)$, which is analogous to the case of the Casson invariant and Rohlin invariant on homology $3$-spheres. Computations of the $\nu$-invariant over Joyce's manifolds were carried out in the work of Scaduto \cite{S20}. 

\subsection{\em Outline}\label{ss1.3} \hfill
 
\vspace{3mm}

In \autoref{s2}, we review necessary materials concerning torsion-free $G_2$-orbifolds. In \autoref{s3}, we set up the frame work for the moduli space of $G_2$-instantons on $G_2$-orbifolds, including deformation, reducibility, and orientation. In \autoref{s4}, we introduce holonomy perturbations following the work of Donaldson, and prove transversality for irreducible moduli spaces. In \autoref{s5}, we prove the compactness of moduli spaces under small perturbations. In \autoref{s6}, we define the invariant $n_{\phi}(X)$ and verify its deformation invariance. Calculations for examples are also included in the end. 

\subsection*{Acknowledgment}
The author is grateful for Simon Donaldson for suggesting this project with patient guidance and encouragement. He also wants to thank Donghao Wang for enlightening discussions.

\section{\large \bf Preliminaries}\label{s2}

In this section, we review the basic materials centering around gauge theory over $G_2$-manifolds, and explain how to extend corresponding notions to the orbifold setting. For a more detailed exposition on the geometry of $G_2$-manifolds, one may consult Joyce's book \cite{J00}. As for the gauge theoretical perspective on $G_2$-manifolds, the paper by Donaldson--Segal \cite{DS11} is an excellent reference. 

\subsection{\em Linear Algebra Related to $G_2$}\label{ss2.1} \hfill
 
\vspace{3mm}

We start with a brief review on the linear algebra behind the Lie group $G_2$. For a more comprehensive treatment, one may consult the notes by Salamon--Walpuski \cite{SW17}. Let $V=(\R^7, g_0)$ be the $7$-dimensional real vector space $\R^7$ equipped with the standard Euclidean inner product $g_0 = \sum_{i=1}^7 dx_i^2$, where $(x_1, ..., x_7)$ are the standard coordinates on $\R^7$. An orientation over $V$ is also fixed so that the corresponding volume form is $\vol_0 = dx_1 \wedge ... \wedge dx_7$. Over $V$, we have a $3$-form 
\begin{equation}\label{e2.1}
\phi_0 := dx_{123} + dx_{145} +dx_{167} + dx_{246} - dx_{257} - dx_{347} - dx_{356},
\end{equation} 
where $dx_{ijk} := dx_i \wedge dx_j \wedge dx_k$. We define $G_2$ to be the subgroup of $GL_{\R}(7)$ consisting of elements that preserve the $3$-form $\phi_0$. It is known that $G_2$ is a compact, connected, simply-connected, semisimple Lie group of real dimension $14$ (c.f. \cite{J00, SW17}). The $3$-form $\phi_0$ defines a cross-product on $V$ via the relation 
\begin{equation}\label{e2.2}
\phi_0(u, v, w) = g_0(u \times v, w), \quad \forall \; u, v, w \in V. 
\end{equation}
Here a cross-product means a skew-symmetric bilinear form $\times: V \times V \to V$ satisfying 
\[
g_0(u \times v, u) = 0 \quad \text{ and } \quad |u \times v|^2 = |u|^2|v|^2 - g_0(u, v)^2, \quad \forall \; u, v \in V. 
\]
If we identify $\R^7$ with the imaginary part of the octonions $\mathbb{O}$, the cross product given by $\phi_0$ is the octonionic multiplication followed by taking imaginary part. The volume form $\vol_0$ of $V$ is related to the $3$-form $\phi_0$ by 
\begin{equation}\label{e2.3}
\iota_u \phi_0 \wedge \iota_v \phi_0 \wedge \phi_0 = 6 g_0(u_0, v_0) \vol_0, \quad \forall \; u, v \in V,
\end{equation}
where $\iota_u$ is the contraction map. (\ref{e2.3}) implies that the $3$-form $\phi_0$ induces the inner product $g_0$ once an orientation of $V$ is chosen.

The induced $G_2$-action on the space of alternating $2$-tensors $\Lambda^2 V^*$ splits the space into the sum of irreducible $G_2$-representations:
\begin{equation}\label{e2.4}
\Lambda^2 V^* := \Lambda^2_7 \oplus \Lambda^2_{14},
\end{equation}
where $\Lambda^2_7$ is isomorphic to the fundamental representation of $G_2$ on $\R^7$, and $\Lambda^2_{14}$ is isomorphic to the adjoint $G_2$-representation on its Lie algebra $\g_2$. The irreducible representations can also be written down explicitly as (c.f. \cite[Theorem 8.5]{SW17}):
\begin{equation}\label{e2.5}
\Lambda^2_7=\left\{\omega \in \Lambda^2V^*: \star(\phi_0 \wedge \omega) =  2 \omega) \right\} \text{ and } \Lambda^2_{14}=\left\{ \omega \in \Lambda^2V^*: \star(\phi_0 \wedge \omega) = - \omega \right\}. 
\end{equation}
In other words, $\Lambda^2_7$ and $\Lambda^2_{14}$ are the $2$- and $-1$-eigenspaces of the operator $\star(\phi_0 \wedge -)$ acting on $\Lambda^2V^*$ respectively. Using the relation 
\[
\star(\phi_0 \wedge \star( \phi_0 \wedge \omega)) = 2\omega + \star(\phi_0 \wedge \omega), \quad \forall \; \omega \in \Lambda^2V^*,
\]
one can write down the projection map as 
\begin{equation}\label{e2.6}
\pi_7(\omega) = \frac{1}{3}( \omega + \star(\phi_0 \wedge \omega)) \text{ and } \pi_{14}(\omega) = \frac{1}{3}(2\omega - \star(\phi_0 \wedge \omega)).  
\end{equation}
The equation (\ref{e2.6}) immediately gives us the following generalization of (\ref{e2.3}):
\begin{equation}\label{e2.7}
\omega \wedge \omega \wedge \phi_0 = \left( 2|\pi_7(\omega)|^2 - |\pi_{14}(\omega)|^2 \right) \vol_0, \quad \forall \; \omega \in \Lambda^2 V^*.  
\end{equation}

We denote by $\mathcal{P}^3(V)$ the orbit of $\phi_0$ under the $GL^+_{\R}(7)$-action, where $GL^+_{\R}(7)$ is the group consisting of orientation-preserving automorphisms of $V$. We refer to an element $\phi \in \mathcal{P}^3(V)$ as a positive $3$-form. One sees immediately that $\mathcal{P}^3(V)$ is an open subset of $\Lambda^3V^*$ which can be identified with $GL^+_{\R}(7) / G_2$. The discussion above tells us that every positive $3$-form $\phi$ determines an inner product $g_{\phi}$ and volume form $\vol_{\phi}$ on an oriented $\R^7$.

\begin{dfn}\label{d2.1}
Let $\phi \in \mathcal{P}^3(V)$ be a positive $3$-form. A oriented $3$-plane $W \subset V$ is said to be $\phi$-associative if $\vol_W = \phi|_W$, where $\vol_W$ is the volume form on $W$ induced from $\g_{\phi}$. 
\end{dfn}
Due to (\ref{e2.2}), every $\phi$-associative subspace $W$ arises as the span of $u, v$ and $u \times_{\phi} v$ for some linearly independent vectors $u, v \in V$. Associative subspaces are closely related to invariant subspaces of $\R^7$ under the action of finite subgroups of $G_2$. For instance, any element $\sigma \in G_2$ of order $2$ fixes a associative subspace in $V$ (c.f. \cite[Proposition 10.8.1]{J00}). In general, the space of invariant vectors in $V$ under the action of a finite subgroup in $G_2$ can be described as follows.

\begin{prop}[{\cite[Proposition 11.1.3]{J00}}]
Let $\Gamma \subset G_2$ be a finite subgroup, and 
\[
V^{\Gamma} :=\{ v \in V: \gamma v = v \; \; \forall \gamma \in \Gamma\}
\]
be the subspace of $\Gamma$-invariant vectors. Then $V^{\Gamma}$ takes one of the following forms.
\begin{enumerate}[label=(\alph*)]
\item $V^{\Gamma}= V \text{ or } \{0\}$. 
\item $V^{\Gamma}$ is an associative subspace, in which case $\Gamma$ is conjugate to a finite subgroup of $SU(2)$ so that $V/\Gamma \simeq \R^3 \times \C^2/\Gamma$. 
\item $V^{\Gamma} \simeq \R$, in which case $\Gamma$ is conjugate to a finite subgroup of $SU(3)$ so that $V/G \simeq \R \times \C^3/\Gamma$. 
\end{enumerate}
\end{prop}

\subsection{\em $G_2$-Orbifolds}\label{ss2.2} \hfill
 
\vspace{3mm}

Following the treatment in \cite{AJY07}, we first review the notion `effective orbifold' in the classical sense \cite{S56}. Let $X$ be a paracompact Hausdorff space. An $n$-dimensional orbifold chart on $X$ is a triple $(\tilde{U}, G, \phi)$ where 
\begin{enumerate}
\item $\tilde{U} \subset \R^n$ is an open neighborhood of the origin; 
\item $G$ is a finite subgroup of $O(n)$; 
\item $\phi:\tilde{U} \to X$ is a $G$-invariant map that induces a homeomorphism $\bar{\phi}: \tilde{U}/G \to \phi(\tilde{U})$. 
\end{enumerate}
An $n$-dimensional orbifold atlas $\mathscr{U} = \{(\tilde{U}_{\alpha}, G_{\alpha}, \varphi_{\alpha})\}$ consists of a family of $n$-dimensional orbifold charts that cover $X$ and satisfy the following compatibility assumption
\begin{itemize}
\item Given two charts $(\tilde{U}_{\alpha}, G_{\alpha}, \varphi_{\alpha})$ and $(\tilde{U}_{\beta}, G_{\alpha}, \varphi_{\beta})$, for any $x \in \varphi_{\alpha}(\tilde{U}_{\alpha}) \cap \varphi_{\beta}(\tilde{U}_{\beta})$ one can find a third chart $(\tilde{U}_{\gamma}, G_{\gamma}, \varphi_{\gamma})$ covering $x$ and smooth embeddings $\lambda_{\alpha\gamma}: \tilde{U}_{\gamma} \hookrightarrow \tilde{U}_{\alpha}$, $\lambda_{\beta\gamma}: \tilde{U}_{\gamma} \hookrightarrow \tilde{U}_{\beta}$ such that 
\[
\varphi_{\gamma} = \varphi_{\alpha} \circ \lambda_{\alpha\gamma} = \varphi_{\beta} \circ \lambda_{\beta\gamma}
\]
\end{itemize}
Given smooth embedding $\lambda_{\beta\alpha}: \tilde{U}_{\alpha} \to \tilde{U}_{\beta}$ satisfying $\varphi_{\alpha} = \varphi_{\beta} \circ \lambda_{\beta\alpha}$, it can be shown \cite{M02} that for each $g_{\alpha} \in G_{\alpha}$ there exists a unique element in $G_{\beta}$, denoted by $\lambda_{\beta\alpha}(g_{\alpha})$, such that 
\[
\lambda_{\beta\alpha} \circ g_{\alpha} = \lambda_{\beta\alpha}(g_{\alpha}) \circ \lambda_{\beta\alpha}.
\]
The assignment $\lambda_{\beta\alpha}: G_{\alpha} \to G_{\beta}$ is injective and makes the embedding $G_{\alpha}$-equivariant. 

\begin{dfn}
An $n$-dimensional orbifold $X$ is a paracompact Hausdorff space equipped with a maximal $n$-dimensional orbifold atlas.  
\end{dfn}

Given a point $x$ in an orbifold $X$, we can choose an orbifold chart $(\tilde{U}, G, \varphi)$ that covers $x$. Let $\tilde{x} \in \varphi^{-1}(x) \subset \tilde{U}$ be a preimage of $x$. We denote by $G_{\tilde{x}}$ the subgroup of $G$ consisting of stabilizers of $\tilde{x}$. Different choices of $\tilde{x}$ in $\varphi^{-1}(x)$ and orbifold charts give rise to canonically isomorphic stabilizers $G_{\tilde{x}}$. Thus we denote by $G_x$ as an abstract group isomorphism to $G_{\tilde{x}}$ for some $\tilde{x} \in \varphi^{-1}(x)$, and refer to $G_x$ as the isotropy group of $x$. Locally near $x$, the orbifold $X$ is modeled on $\tilde{U}_x/G_x$ where $\tilde{U}_x \subset \R^n$ is an open set so that $x$ is identified with the origin under the quotient. We denote by 
\begin{equation}
S_X:= \{ x \in X: G_x \neq 1\}
\end{equation}
the set of points with non-trivial isotropy group. Each point in $S_X$ is called a singular point of the orbifold $X$. 

Given an orbifold chart $(\tilde{U}_{\alpha}, G_{\alpha}, \varphi_{\alpha})$, one can consider the tangent bundle $T\tilde{U}_{\alpha}$ which is naturally a $G_{\alpha}$-equivariant bundle with the $G_{\alpha}$-action induced from differentiating the $G_{\alpha}$-action on the base $\tilde{U}_{\alpha}$. The orbifold charts $(T\tilde{U}_{\alpha}, G_{\alpha})$ can be patched together to get 
\[
TX:= \bigsqcup_{\alpha} (T\tilde{U}_{\alpha}/G_{\alpha}) / \sim,
\] 
where the patching map is given by the differential of the embeddings $\lambda_{\beta\alpha}$. The total space $TX$ is a $2n$-dimensional orbifold, which we refer to as the tangent bundle of $X$. There is a natural projection $\pi: TX \to X$ whose fiber at $x \in X$ is 
\[
\pi^{-1}(x) = T_{\tilde{x}}\tilde{U} / G_{\tilde{x}},
\]
where $\tilde{x} \in \phi^{-1}(x)$ is a point in the preimage of $x$.

Such a construction can be carried over to the cotangent bundle $T^*X$, the symmetric tensor bundle $\Sym^* T^*X$, and the alternating tensor bundle $\Lambda^* T^*X$. Locally over a chart $(\tilde{U}_{\alpha}, G_{\alpha})$ these bundles are $G_{\alpha}$-equivariant bundles. A global section of these bundles consists of locally $G_{\alpha}$-equivariant sections that are compatible under the patching maps. In particular, a Riemannian metric $g$ over an orbifold $X$ is a positive-definite section of $\Sym^2 T^*X$. Analogous to the manifold scenario, an orbifold Riemannian metric $g$ gives rise to a Levi--Civita connection on all tensor bundles of $X$, which we denote by $\nabla_g$. 

We say $X$ is orientable if $\Lambda^n T^*X$ admits a nowhere vanishing section. Such a section specifies an orientation of $X$. Combining with the Riemannian metric $g$, we get a volume form $\vol$ over $X$. The integration over $X$ is defined as follows. Locally over $U_{\alpha} = \varphi_{\alpha}(\tilde{U}_{\alpha}) \subset X$, a compactly supported $n$-form $\omega_{\alpha}$ on $U_{\alpha}$ arises as the $G_{\alpha}$-quotient of an equivariant $n$-form $\tilde{\omega}_{\alpha}$. Then we define 
\[
\int_{U_{\alpha}} \omega_{\alpha}:= \frac{1}{|G_{\alpha}|} \int_{\tilde{U}_{\alpha}} \tilde{\omega}_{\alpha},
\]
where the integral on the RHS is defined with the help of the volume form. To integrate a general $n$-form over $X$, we take a partition of unity subordinate to $\{U_{\alpha}\}$ to truncate the $n$-form, and add up integrals on the charts. 

Let $(\tilde{U}_{\alpha}, G_{\alpha})$ be an oriented orbifold chart. For each $\tilde{x} \in \tilde{U}_{\alpha}$, we denote by $\mathcal{P}^3(T_{\tilde{x}}\tilde{U}_{\alpha})$ the set of positive $3$-forms on $T_{\tilde{x}}\tilde{U}_{\alpha}$. We denote by $\mathcal{P}^3(T\tilde{U}_{\alpha}) = \bigcup_{\tilde{x}} \mathcal{P}^3(T_{\tilde{x}}\tilde{U}_{\alpha})$. 

\begin{dfn}
Let $X$ be an oriented $7$-dimensional orbifold. 
\begin{itemize}[leftmargin=*]
\item An almost $G_2$-structure over $X$ is a $3$-form $\phi \in C^{\infty}(X, \Lambda^3T^*X)$ whose restriction to each chart $(\tilde{U}_{\alpha}, G_{\alpha})$ is the $G_{\alpha}$-quotient of an equivariant positive $3$-form $\tilde{\phi}_{\alpha} \in C^{\infty}(\tilde{U}_{\alpha}, \mathcal{P}^3(T\tilde{U}_{\alpha}))$. 
\item Denote by $g_{\phi}$ the induced Riemannian metric from an almost $G_2$-structure $\phi$. We say $\phi$ is a (torsion-free) $G_2$-structure if it is $g_{\phi}$-parallel, i.e. $\nabla_{g_{\phi}} \phi = 0$. 
\end{itemize}
\end{dfn}

The condition $\nabla_{g_{\phi}} \phi = 0$ is equivalent to that the Levi-Civita connection $\nabla_{g_{\phi}}$ is torsion-free. The following result provides us with a useful criterion to justify whether an almost $G_2$-structure is parallel or not. 

\begin{lem}[{\cite[Proposition 10.1.3]{J00}}]
Let $\phi$ be an almost $G_2$-structure over an orbifold $X$. Then $\phi$ is $g_{\phi}$-parallel if and only if 
\[
d\phi = 0 \quad \text{ and } \quad d\star_{\phi} \phi = 0,
\]
where $\star_{\phi}$ is the Hodge star given by the induced metric $g_{\phi}$. 
\end{lem}

Let $(\tilde{U}_{\alpha}, G_{\alpha}, \varphi_{\alpha})$ be an orbifold chart of $X$, and $U_{\alpha} = \varphi_{\alpha}(\tilde{U}_{\alpha}) \subset X$. In the definition, the restriction of a $G_2$-structure $\phi|_{U_{\alpha}}$ is given by the quotient of a $G_{\alpha}$-equivariant positive $3$-form $\tilde{\phi}_{\alpha}$. Since $G_{\alpha}$ fixes the origin in $\tilde{U}_{\alpha}$, it preserves the positive $3$-form $\tilde{\phi}_{\alpha}|_0$. Thus, up to $GL^+_{\R}(7)$-conjugation, $G_{\alpha}$ can be identified with a discrete subgroup of $G_2$. The conjugacy classes of finite subgroups of $G_2$ have been classified in \cite{CW83, G95}. We shall not dive into the details here. Roughly speaking, there are seven finite subgroups in $G_2$ that act irreducibly on $\R^7$ through the fundamental representation of $G_2$. All other finite subgroups of $G_2$ lie in either $SU(2) \times SU(2)$ or $SU(3)$, up to conjugation. So the isotropy group $G_x$ of a point $x \in X$ has to be one of those. 

Typical examples of compact $G_2$-orbifolds arise as the global quotient of a compact $G_2$-manifold under a finite group action that preserves the $G_2$-structure. We briefly review some constructions here. For more details, see Joyce's book \cite{J00}. 

\begin{exm}
Let $M$ be a compact $7$-manifold equipped with a $G_2$-structure $\phi$. \cite[Theorem 10.2.1]{J00} tells us that the connected component of the holonomy group $\Hol^0(g_{\phi})$ is one of $\{1\}$, $SU(2)$, $SU(3)$, or $G_2$. 
\begin{enumerate}[leftmargin=*,label=(\alph*)]
\item When $\Hol^0(g_{\phi}) = \{1\}$, we consider the flat $7$-torus $T^7 = \R^7/\Z^7$. The $3$-form $\phi_0$ (\ref{e2.1}) over $\R^7$ is invariant under the $\Z^7$-translation. Thus it descends to a torsion-free $G_2$-structure $\phi_0$ over $T^7$. Let $\Gamma \subset G_2$ be a finite subgroup. Since $\Gamma$ preserves the $3$-form $\phi_0$. The quotient $T^7/\Gamma$ is a $G_2$-orbifold. 

\item When $\Hol^0(g_{\phi}) = SU(2)$, we consider the product $M=T^3 \times K3$ between the flat $3$-torus and the $K3$ surface. The $K3$ surface admits a complex structure $J$ with compatible Kähler form $\omega$ and holomorphic volume form $\theta$. Then we have a $G_2$-structure on $M$ (c.f. \cite[Proposition 11.1.1]{J00})
\begin{equation}
\phi = dx_{123} + dx_1 \wedge \omega + dx_2 \wedge \Rea \theta - dx_3 \wedge \Ima \theta.
\end{equation}
Then one can take a finite group $\Gamma$ of diffeomorphisms that preserve $\phi$. The global quotient $M/\Gamma$ is a $G_2$-orbifold. For instance, $\Z/2$ acts on $T^3 \times K3$ via the the map $(x_1, x_2, x_3) \mapsto (x_1, -x_2, -x_3)$ on $T^3$ and complex conjugation on $K3$. 

\item When $\Hol^0(g_{\phi}) = SU(3)$, we consider the product $M= S^1 \times Y$ between the circle and a Calabi--Yau $3$-fold $Y$. Let $\omega$ be the Kähler form and $\theta$ be the holomorphic volume form over $Y$. Then a $G_2$-structure on $M$ takes the form (c.f. \cite[Proposition 11.1.2]{J00})
\begin{equation}
\phi = dx \wedge \omega + \Rea \theta. 
\end{equation}
Then $\Z/2$ acts on $S^1 \times Y$ via the map obtained by the antipodal map on $S^1$ and complex conjugation on $Y$. 
\end{enumerate}
\end{exm}

Motivated by these examples, we shall consider $G_2$-orbifolds with certain explicit local models near its singular set.

\begin{dfn}\label{d2.7}
A compact $G_2$-orbifold $X$ is called simple if its singular set $S_X$ is locally modeled on one of the following types.
\begin{enumerate}
\item[(I)] $x \in S_X$ admits an orbifold chart $(\tilde{U}_x, \Z/2)$ equipped with the standard positive $3$-form $\phi_0$, where the generator $-1$ of $\Z/2 = \{1, -1\}$ acts on $\tilde{U}_x \subset \R^7$ by 
\[
(x_1, x_2, x_3, x_4, x_5, x_6, x_7) \longmapsto (x_1, x_2, x_3, -x_4, -x_5, -x_6, -x_7).
\] 
\item[(II)] $x \in S_X$ admits an orbifold chart $(\tilde{U}_x, \Z/2 \oplus \Z/2)$ equipped with the standard positive $3$-form $\phi_0$, where the generators $(1,-1)$ and $(-1,1)$ act on $\tilde{U}_x \subset \R^7$ by 
\[
\begin{split}
(1,-1) \cdot (x_1, x_2, x_3, x_4, x_5, x_6, x_7)  & = (x_1, x_2, x_3, -x_4, -x_5, -x_6, -x_7) \\
(-1,1) \cdot (x_1, x_2, x_3, x_4, x_5, x_6, x_7)  & = (x_1, -x_2, -x_3, x_4, x_5, -x_6, -x_7). 
\end{split}
\]
\end{enumerate}
\end{dfn}

The first local model (1) in \autoref{d2.7} describes the singular set locally as the quotient of an associative subspace in $\R^7$ whose neighborhood arise an $\R P^3$-cone bundle. The second local model (2) in \autoref{d2.7} describes that two pieces of the singular set modeled on (1) intersect each other into a curve. In Joyce's construction of compact $G_2$-manifolds, $G_2$-orbifolds arise naturally as global quotients of the flat $7$-torus under a finite group action. For the concrete examples presented in his paper \cite{J96}, most of the relevant $G_2$-orbifolds are simple.

Now we briefly discuss the homotopy and (co)homology groups of orbifolds following the treatment in \cite{AJY07}. To each orbifold $X$, one can associate a classifying space $c:\mathcal{B}X \to X$ which has the homotopy type of a CW-complex \cite{S74}. This classifying space is unique up to weak homotopy equivalent. Then we define the homotopy groups and (co)homology groups of the orbifold $X$ respectively as 
\begin{equation}
\pi^{\orb}_n(X, x) := \pi_n(\mathcal{B}X, x_c) \quad  H^{\orb}_n(X; R) :=  H_n(\mathcal{B}X; R) \quad H^n_{\orb}(X;R):=H^n(\mathcal{B}X; R),
\end{equation}
where $x_c \in \mathcal{B}X$ satisfies $c(x_c) = x$, and $R$ is any commutative ring. It is proved \cite[Corollary 1.24]{AJY07} that every (classical) compact $n$-orbifold $X$ is the global quotient of a smooth manifold $M$ under a smooth, effective, and almost free action by a compact Lie group $G$. Indeed, one can take $M$ to be the total space of the frame bundle of $X$ and $G=O(n)$. In this case the classifying space $\mathcal{B}X$ arises as the Borel construction $EG \times_G M$, where $EG$ is the classifying bundle of the compact Lie group $G$. So the orbifold (co)homology of $X=M/G$ is the equivariant (co)homology of $(M, G)$. 

One can also consider the deRham cohomology over an orbifold whose chain complex consists of orbifold $n$-forms $\Omega^n(X)$, and whose differential is the exterior derivative on forms. We denote the deRham cohomology by $H^*_{\dR}(X)$ which can be identified with the (non-orbifold) singular cohomology $H^*(X; \R)$ of $X$ with $\R$-coefficient (c.f. \cite{S56}). Moreover, the orbifold cohomology $H^*_{\orb}(X;\R)$ is also isomorphic to usual singular cohomology $H^*(X; \R)$.

When $G$ is a finite group, the fibration $M \to EG \times_G M \to BG$ gives rise to a short exact sequence:
\begin{equation}
1 \longrightarrow \pi_1(M) \longrightarrow \pi^{\orb}_1(X) \longrightarrow G \longrightarrow 1. 
\end{equation}
By passing to the universal cover $\tilde{M}$ of $M$, the orbifold fundamental group $\pi_1^{\orb}(X)$ can be identified with the group of covering transformations on $\tilde{M}$ \cite[Theorem 2.18]{AJY07}. The orbifold fundamental group has a more elementary description which corresponds to the definition of the fundamental group of manifolds. We recall this description following the exposition in \cite{H90}. 

\begin{dfn}
Let $X$ be an orbifold equipped with an orbifold atlas $\{(\tilde{U}_{\alpha}, G_{\alpha}, \varphi_{\alpha})\}_{\alpha \in I}$. Let $x, y \in X$ be two fixed points.
\begin{itemize}[leftmargin=*]
\item An orbipath $\gamma$ from $x$ to $y$ consists of the following data.
\begin{enumerate}
\item A partition $0=t_0 \leq t_1 \leq ... \leq t_n=1$ of the unit interval $[0, 1]$. 
\item An index function $\tau: \{1, ..., n\} \to I$ and a family of continuous paths $\tilde{\gamma}_i: [t_{i-1}, t_i] \to \tilde{U}_{\tau(i)}$ such that 
\[
\varphi_{\tau(1)} (\tilde{\gamma}_1(0)) = x, \quad \varphi_{\tau(i)}(\tilde{\gamma}_i(t_{i-1}))= \varphi_{\tau(i-1)}(\tilde{\gamma}_{i-1}(t_{i-1})), \quad \varphi_{\tau(n)}(\tilde{\gamma}_n(1)) = y. 
\]
\end{enumerate}

\item Two orbipaths $\gamma$ and $\gamma'$ from $x$ to $y$ are equivalent if, after passing to subdivisions of $[0,1]$ and refinements to a common index function $\tau$, the local lifts are related by $\tilde{\gamma}_i = g_i \cdot \tilde{\gamma}'_i$ for some $g_i \in G_{\tau(i)}$. 

\item Two orbipaths $\gamma$ and $\gamma'$ are homotopic relative to ends if one can find a continuous family of orbipaths $(\gamma^s)_{s \in [0,1]}$ so that for all $s \in [0,1]$ one has 
\[
\gamma^s(0) = x, \quad \gamma^0 = \gamma, \quad, \gamma^1 = \gamma', \quad \gamma^s(1) = y. 
\]

\item The orbifold fundamental group $\pi^{\orb}_1(X, x)$ is defined to the set of orbipaths based at $x$, up to orbipath equivalence and homotopy relative to ends, which is equipped with the obvious group structure. 
\end{itemize}
\end{dfn}

One can heuristically think of an orbipath as a continuous path on $X$ with chosen local lifts, and the orbifold fundamental group as homotopy classes of such objects. The notion orbipath will be useful when we consider holonomy of connections over orbibundles. 

\subsection{\em Orbibundles and Connections} \label{ss2.3} \hfill
 
\vspace{3mm}

\begin{dfn}
Let $X$ be an $n$-orbifold equipped with an orbifold atlas $\{(\tilde{U}_{\alpha}, G_{\alpha}, \varphi_{\alpha})\}$. 
\begin{itemize}[leftmargin=*]
\item A rank-$r$ vector orbibundle $E$ over $X$ is a collection $\{\tilde{E}_{\alpha}\}$ of $G_{\alpha}$-equivariant rank-$r$ bundles $\tilde{E}_{\alpha} \to \tilde{U}_{\alpha}$ equipped with fiberwise isomorphic embeddings $\tilde{E}_{\alpha} \hookrightarrow \tilde{E}_{\beta}$ that are compatible with the embeddings $\tilde{U}_{\alpha} \hookrightarrow \tilde{U}_{\beta}$. 

\item Let $G$ be a compact Lie group. A principal $G$-orbibundle $P$ over $X$ is a collection $\{\tilde{P}_{\alpha}\}$ of $G_{\alpha}$-equivariant principal $G$-bundles $\pi_{\alpha}: \tilde{P}_{\alpha} \to \tilde{U}_{\alpha}$ (the left $G_{\alpha}$-action commutes with the right $G$-action on $\tilde{P}_{\alpha}$) equipped with fiberwise isomorphic embeddings $\tilde{P}_{\alpha} \hookrightarrow \tilde{P}_{\beta}$ that are compatible with the embeddings $\tilde{U}_{\alpha} \hookrightarrow \tilde{U}_{\beta}$. 
\end{itemize}
\end{dfn}

Let $P$ be a principal $G$-orbibundle over $X$. Same as the case of manifolds \cite{KN63}, a connection $A$ over $P$ is defined as a compatible collection of $G$-equivariant splittings of the $(G_{\alpha}, G)$-equivariant bundle sequence:
\begin{equation}\label{e2.13}
0 \longrightarrow \tilde{P}_{\alpha} \times \g \xrightarrow{\psi_{\alpha}} T\tilde{P}_{\alpha} \xrightarrow{\pi_{\alpha, *}} \pi_{\alpha}^*T\tilde{U}_{\alpha} \longrightarrow 0,
\end{equation}
where $\g$ is the Lie algebra of $G$, and $\psi_{\alpha}$ is the infinitesimal Lie algebra action that induces the isomorphism between $\g$ and the vertical fiber of $T\tilde{P}_{\alpha}$. A splitting of (\ref{e2.13}) can be realized either as a left inverse of $\psi_{\alpha}$ which gives us a $(G_{\alpha}, G)$-equivariant $\g$-valued $1$-form on $\tilde{P}_{\alpha}$ after composing the projection onto $\g$, or as a right inverse of $\pi_{\alpha, *}$ whose image is a $(G_{\alpha}, G)$-equivariant horizontal subspace of $T\tilde{P}_{\alpha}$. 

To a representation $\rho: G \to \Aut(V)$, we can associate a vector orbibundle $E = P \times_{\rho} V$. In particular, the adjoint representation $\ad: G \to \Aut(\g)$ gives us an adjoint bundle $\g_P:= P \times_{\ad} \g$. The space $\A_P$ of smooth connections over $P$ is then an affine space modeled on $\Omega^1(X, \g_P)$. The adjoint action $\Ad: G \to \Aut(G)$ results in another fiber bundle $G_P:= P \times_{\Ad} G$. Although locally the fiber $G_P$ is the Lie group $G$, it is not a principal $G$-bundle. The automorphism group of $P$ is then identified as $C^{\infty}(X, G_P)$ whose Lie algebra is $\Omega^0(X, \g_P)$. We refer to $\G_P:=C^{\infty}(X, G_P)$ as the gauge group of $P$, and an automorphism of $P$ as a gauge transformation. Then $\G_P$ acts on $\A_P$ via pull-backs:
\begin{equation}
u \cdot A :=uAu^{-1}= A - (d_Au)u^{-1}. 
\end{equation}

Given a $G$-connection $A$ on a principal $G$-orbibundle $P$, we can define the holonomy of $A$ around orbiloops in $X$ similarly as in the case of manifolds. When $A$ is flat, i.e. its curvature $F_A = 0$, the holonomy is invariant under homotopy of orbiloops. After fixing a basepoint $x \in X$ and a lift $p \in P$ of $x$, the holonomy of a flat connection defines a representation 
\begin{equation}
\Hol(A) : \pi_1^{\orb}(X, x) \longrightarrow G.  
\end{equation}
Different choices of $x$ and $p$ changes the representation by the adjoint action of an element in $G$. Conversely, given a representation $\rho \in \Hom(\pi_1^{\orb}(X), G)$, one can build a principal $G$-orbibundle $P_{\rho}$ with a connection $A_{\rho}$ that realizes $\rho$ via its holonomy (c.f. \cite{SY18}). So we consider the set $\mathfrak{F}(X, G)$ of pairs $(P, A)$ where $P$ is a principal $G$-orbibundle over $X$, and $A$ is a flat connection over $P$. Two pairs $(P, A)$ and $(P', A')$ are equivalent if there exists a bundle isomorphism $f: P \to P'$ so that $f^*A' = A$. Then one can prove the above procedure establishes a one-to-one correspondence between the equivalence classes of such pairs and conjugacy classes of $\pi_1^{\orb}(X)$-representations over $G$.

\begin{prop}[{\cite{SY18}}]
The holonomy map induces a one-to-one correspondence:
\[
\mathfrak{F}(X, G)/ \sim \quad  \longleftrightarrow \quad  \Hom(\pi^{\orb}_1(X), G)/\Ad. 
\]
\end{prop} 

Such a correspondence is well-known in the manifold case. One may consult \cite[Chapter 13]{T11} for a complete proof. 

Suppose now that $G$ is a Lie subgroup of $O(n)$ for some positive integer $n$. Then the representation $\rho: G \to \Aut(\R^n)$ associates a rank-$n$ bundle $E:=P \times_{\rho} X$ to a principal $G$-orbibundle $P \to X$. Regarding the total space of $E$ as an orbifold, its classifying space $\mathcal{B}E$ is naturally an $\R^n$-bundle over $\mathcal{B}X$ (c.f. \cite{AJY07}). Then we define the Steifel--Whitney classses and Pontryagin classes of $P$ as the characteristic classes associated to the bundle $\mathcal{B}E \to \mathcal{B}X$:
\begin{equation}
w_i(P) \in H^i_{\orb}(X; \Z/2) \quad \text{ and } \quad p_i(P) \in H^{4i}_{\orb}(X; \Z). 
\end{equation}
If $G$ is a Lie subgroup of $U(n)$, Chern classes are defined similar for principal $G$-orbibundles. The Chern--Weil theory that relates characteristic forms to the curvature form of connections on a principal $G$-orbibundle are also developed. For a thorough treatment, one may consult \cite{LTX07}. When the orbifold $X$ is given by a global quotient $M/G$, and $P$ is an equivariant $G$-bundle, the book by Berline--Getzler--Vergne \cite[Chapter 7]{BGV92} contains an excellent exposition on the equivariant Chern--Weil theory. For instance, under the identification $H^*_{\dR}(X) \simeq H^*_{\orb}(X; \R)$ the real Pontryagin classes of $P$ are given by the same combination of the symmetric polynomials of the curvature form $F_A$ for a $G$-connection $A$ over $P$ as in the case of principal $G$-bundles over manifolds.

\section{\large \bf Gauge Theory over $G_2$-Orbifolds}\label{s3}

In this section, we set up the framework of gauge theory over $G_2$-orbifolds following Donaldson--Segal \cite{DS11}. For a more comprehensive treatment, one may consult the thesis of Sá Earp \cite{SE09}. 

To start, we pin-down the convention of orbifolds and bundles. Let $(X, \phi)$ be a compact orbifold equipped with a torsion-free $G_2$-structure, and $\pi: P^{\dagger} \to X$ a principal $U(r)$-bundle.  From $P^{\dagger}$, we obtain a line bundle $\det P^{\dagger}$ via the determinant homomorphism $\det: U(r) \to U(1)$ and a $PU(r)$-bundle $P$ via the projective homomorphism $\pr: U(r) \to PU(r)$. A general connection of $P^{\dagger}$ is denoted by $A^{\dagger}$, which induces a connection $\theta_A$ on $\det P^{\dagger}$ and a connection $A$ on $P$. 

\subsection{$G_2$-Instantons} \hfill

\vspace{3mm}

We fix a smooth connection $A^{\dagger}_0$ on $P^{\dagger}$ as a reference connection. We also fix an integer $k \geq 4$ to ensure the Soboleve embedding $L^2_k \hookrightarrow C^0$. Then we get an induced connection $A$ on the $PU(r)$-bundle $P$. Denote by $\A_k$ the space of all $PU(r)$-connections over $P$ in class $L^2_k$. Explicitly $A \in \A_k$ if and only if 
\[
\|A - A_0\|_{L^2_{k, A_0}(X)} := \sum_{i=0}^k \left(\int_X  |\nabla^i_{A_0}(A - A_0)|^2 d\vol \right)^{\frac{1}{2}} < \infty,
\]
where $\nabla^i_{A_0}$ means to take the $i$-th derivative with respect to $A_0$. Then $\A_k$ is an affine space modeled on $L^2_k(X, \Lambda^1 \otimes \g_P)$, where $\Lambda^*$ denotes $\Lambda^* T^*X$. 

\begin{rem}
We always identify $\mathfrak{pu}(r)$ with the Lie algebra $\mathfrak{su}(r)$ via the isomorphism $\ad: \su(r) \to \mathfrak{pu}(r)$ given by differentiating the adjoint map. The inner product on $\mathfrak{su}(r)$ is defined via the Killing form:
\[
\langle \alpha, \beta \rangle:= -\tr(\ad(\alpha) \circ \ad(\beta)) = -2r \tr(\alpha\beta), \quad \forall \alpha, \beta \in \su(r). 
\]
\end{rem}

The gauge group $\G_{k+1}$ is defined as follows. Note that the adjoint action of $U(r)$ on $SU(r)$ factors through $PU(r)$. So we can consider bundles 
\[
G'_P:=P \times_{\Ad} SU(r) \qquad G_P:= P \times_{\Ad} PU(r).
\]
Sections of the second bundle $G_P$ are automorphisms of $P$. The projective homomorphism $\pr: SU(r) \to PU(r)$ induces a map $\pr_P: G'_P \to G_P$. We define the gauge group to be 
\[
\G_{k+1}:= L^2_{k+1}(G'_P)
\]
which acts on the bundle $P$ by pre-composing with $\pr_P$. Elements in $\G_{k+1}$ are commonly referred to as `determinant-1' of $SU(r)$ gauge transformations. Sections of $G'_P$ act on the space of connections via pull-backs as usual:
\begin{equation}
u \cdot A := A - (d_A u )u^{-1}, \quad u \in \G_{k+1}, \; A \in \A_k. 
\end{equation}
Over $7$-orbifolds, the Sobolev multiplication theorem $L^2_{k+1} \times L^2_k \hookrightarrow L^2_k$ holds when $k \geq 3$. Thus gauge transformations in $\G_{k+1}$ preserves $\A_k$. The Lie algebra of $\G_{k+1}$ is identified with $L^2_{k+1}(X, \g_P)$.

Alternatively, one can work with $U(r)$-connections over $P^{\dagger}$ with an extra constraint. Denote by $\theta_0=\det A^{\dagger}_0$ the induced connection on the determinant line $\det P^{\dagger}$. Then the space $\A_k$ of connections can be identified with the space of $L^2_k$ $U(r)$-connections $A^{\dagger}$ on $P^{\dagger}$ such that $\det A^{\dagger} = \theta_0$. Gauge transformations on $P^{\dagger}$ that preserve $\theta_0$ have trivial determinants, thus form the aforementioned group $\G_{k+1}$.

Denote by $\psi:= \star_{\phi} \phi$ the $4$-form associated to the positive $3$-form $\phi$. The $G_2$-variant of the Chern-Simons functional is defined as:
\begin{equation}
\cs_{\phi}(A_0 + a):= -2r \int_X \left( \frac{1}{2} \tr( a \wedge F_{A_0+a})+ \frac{1}{3} \tr( a \wedge a \wedge a) \right) \wedge \psi.
\end{equation}
It's clear that the Chern--Simons functional defines an analytic functional over the space of connections $\cs_{\phi}: \A_k \to \R$. 

\begin{lem}\label{l3.1}
The gradient of the Chern--Simons functional is given by
\[
\grad \cs_{\phi}(A) = \star_{\phi} (F_{A} \wedge \psi). 
\]
\end{lem}

\begin{proof}
We identify the tangent space $T_{A}\A_k$ at $A=A_0+a$ by $L^2_k(X, \Lambda^1 \otimes \g_P)$. Then the differential can be computed as 
\[
\begin{split}
\frac{1}{2r} D \cs_{\phi}|_{A} (\alpha) = & - \int_X \frac{1}{2} \left(\tr(a \wedge d_{A_0} \alpha) + \tr(\alpha \wedge d_{A_0}a )\right) \wedge \psi \\
& - \int_X  \frac{1}{3} \left(\tr(\alpha \wedge a \wedge a) + \tr(a \wedge \alpha \wedge a) + \tr(a \wedge a \wedge \alpha) \right) \wedge \psi. 
\end{split}
\]
Integration by parts gives us 
\[
\int_X \tr(a \wedge d_{A_0} \alpha) \wedge \psi= \int_X \tr(\alpha \wedge d_{A_0} a) \wedge \psi. 
\]
The cyclic-permutation invariance of the trace operator tells us that 
\[
\begin{split}
D \cs_{\phi}|_{A} (\alpha) & = - 2r \int_X \tr(\alpha \wedge (d_{A_0} a + a\wedge a))  \wedge \psi \\
& = - 2r \int_X \tr(\alpha \wedge F_{A} \wedge \psi) \\
& = \langle \alpha, \star_{\phi} (F_{A} \wedge \psi) \rangle_{L^2}.
\end{split}
\]
This completes the proof. 
\end{proof}

\begin{dfn}
Let $P^{\dagger} \to X$ be a principal $U(r)$-orbibundle over a $G_2$-orbifold $(X, \phi)$. A unitary connection $A^{\dagger}$ over $P^{\dagger}$ is a projective $G_2$-instanton if 
\[
F_{A} \wedge \psi = 0,
\]
where $A$ is the induced connection of $A^{\dagger}$ on the $PU(r)$-bundle $P$. 
\end{dfn}


As we work with connections on the $PU(r)$-bundle $P$ in most times, we shall simply refer to $A$ as a $G_2$-instanton if $F_A \wedge \psi = 0$. \autoref{l3.1} tells us that a $G_2$-instanton is a critical point of the Chern--Simons functional $\cs_{\phi}$. This perspective of $G_2$-instantons arises as a reflection on the scenario in dimension 3 where flat connections are critical points of the Chern--Simons functional. On the other hand, $G_2$-instantons also bear similarities with ASD instantons in dimension $4$. Recall in (\ref{e2.4}) that the $G_2$-action on $\Lambda^2(\R^7)^*$ decomposes into two $G_2$-representations $\Lambda^2_7 \oplus \Lambda^2_{14}$. Since the $3$-form $\phi$ is parallel, such a decomposition carries through to the bundle of $2$-forms over the $G_2$-orbifold $X$ immediately:
\[
\Lambda^2T^*X = \Lambda^2_7 \oplus \Lambda^2_{14}. 
\]
Denote by $\pi_7: \Lambda^2 \to \Lambda^2_7$ the orthogonal projection. For simplicity, we shall write $\star$ for $\star_{\phi}$ when the context is clear. 

\begin{lem}\label{l3.3}
Let $P \to X$ be a principal $U(r)$-orbibundle over a $G_2$-orbifold $(X, \phi)$. Then the following are equivalent.
\begin{enumerate}[label=(\alph*)]
\item $A$ is a $G_2$-instanton over $P$.
\item $\pi_7(F_{A}) = 0$.
\item $\star_{\phi}(F_{A} \wedge \phi) = -F_{A}$. 
\end{enumerate}
\end{lem}

\begin{proof}
The equivalence between (b) and (c) follows from (\ref{e2.5}). For the equivalence between (a) and (c), we note the following identity (c.f. \cite[Equation (4.60)]{SW17})
\[
\star(\psi \wedge \star(\psi \wedge F_{A})) = F_{A} + \star(F_{A} \wedge \phi). 
\]
Since the adjoint of $\psi \wedge -$ is $\star (\psi \wedge (\star -))$, we conclude that $F_{A} + \star(F_{A} \wedge \phi) = 0$ if and only if $\psi \wedge F_{A} = 0$. This completes the proof. 
\end{proof}

The second criterion (b) in \autoref{l3.3} takes $G_2$-instantons to the analogous setting of ASD instantons in dimension four. Recall $SO(4)$ admits two irreducible $3$-dimensional representations, which decomposes the bundle of $2$-forms over a $4$-manifold into two parts $\Lambda^2 = \Lambda^+ \oplus \Lambda^-$ consisting of the self-dual and anti-self-dual $2$-forms respectively. Then a connection $A$ is ASD if its self-dual part vanishes, i.e. $\pi_+(F_{A}) = 0$. The analogy with ASD instantons can be brought further by considering the Yang--Mills functional
\begin{equation}
\begin{split}
\mathcal{YM}: \A_k & \longrightarrow \R \\
A & \longmapsto \int_X |F_{A}|^2 d\vol. 
\end{split}
\end{equation}
When $A$ is a $G_2$-instanton, (\ref{e2.6}) and  (\ref{e2.7}) imply that
\begin{equation}\label{e3.3}
\mathcal{YM}(A) = \|\pi_{14}(F_{A})\|^2_{L^2} = -\int_X 2r \tr(F_{A} \wedge F_{A}) \wedge \phi =  8\pi^2 \langle p_1(\g_P) \smile [\phi], [X] \rangle, 
\end{equation}
where $p_1(\g_P) \in H^4_{\orb}(X; \R)$ is the real Pontryagin class of the orbibundle $\g_P$. Since the RHS of (\ref{e3.3}) is independent of the choice of the connection $A$, $G_2$-instantons are the absolute minima of the Yang--Mills functional. This property also characterizes ASD instantons in dimension four. 

\begin{dfn}
The moduli space of projective $G_2$-instantons over $P^{\dagger}$ is defined to be 
\[
\M_{\phi}(X, P^{\dagger}):= \left\{ A \in \A_k : F_{A} \wedge \psi =0 \right\}/\G_{k+1}. 
\]
\end{dfn}

Although we work with connections in $L^2_k$ class, each equivalence class of $G_2$-instantons admit smooth representatives. So we choose to omit regularity in the notation of moduli spaces. 

\subsection{\em Local Deformation} \label{ss3.2} \hfill

\vspace{3mm}

The local structure of the moduli space $\M_{\phi}(X, P^{\dagger})$ is best understood via the study of the `deformation complex'. Infinitesimally, the action of the gauge group $\G_{k+1}$ on $\A_k$ is given by 
\[
\evalat*{{d \over dt}}{t=0} e^{t\xi} \cdot A = - d_A \xi, \quad \xi \in L^2_{k+1}(X, \g_P)
\]
Moreover, the linearization of the $G_2$-instanton operator $F_A \wedge \psi$ is given by 
\[
\evalat*{{d \over dt}}{t=0} F_{A+t\alpha} \wedge \psi = d_A \alpha \wedge \psi, \quad \alpha \in L^2_k(X, \Lambda^1 \otimes \g_P). 
\]
This leads us to consider the following sequence $E_A(X)$ of differential operators:
\[
L^2_{k+1}(X, \g_P) \xrightarrow{-d_A} L^2_k(X, \Lambda^1 \otimes \g_P) \xrightarrow{d_A \wedge \psi} L^2_{k-1}(X, \Lambda^6 \otimes \g_P) \xrightarrow{d_A} L^2_{k-2}(X, \Lambda^7 \otimes \g_P). 
\]
We note that the sequence $E_A(X)$ becomes a chain complex if and only if $A$ is a $G_2$-instanton, in which case we refer to $E_A(X)$ as the deformation complex of the moduli space $\M_{\phi}(X, P^{\dagger})$ at $[A]$. This complex $E_A(X)$ is an elliptic complex in the sense that the associated sequence given by the principal symbols of the differential operators are exact (c.f. \cite{DS11}). So the cohomology of $E_A(X)$ is finite dimensional. 

The cohomology of $E_A(X)$ has concrete geometric meanings similar to the deformation complex of $ASD$ instantons in dimension four. $H^0(E_A(X))$ is the Lie algebra of the isotropy group $\Stab(A)$ of $A$. To better understand $\Stab(A)$, we consider the corresponding connection $A^{\dagger}$ on $P^{\dagger}$. Let's fix a non-singular basepoint $x_0 \in X$, and identify the fiber of $G'_P|_{x_0}$ with $SU(r)$. $u \in \Stab(A)$ means that $u$ is $A^{\dagger}$-parallel, thus is determined by its restriction to $x_0$. Thus $\Stab(A)$ is identified with a subgroup of $SU(r)$ (actually it's a Lie subgroup although not necessarily closed). Denote by $\Hol(A^{\dagger}, x_0) \subset U(r)$ the holonomy group generated by automorphisms of $P^{\dagger}|_{x_0}$ obtained via $A^{\dagger}$-parallel transports along orbi-loops based at $x_0$. Then $\Stab(A)$ is identified with  the centralizer of $\Hol(A^{\dagger}, x_0)$ in $SU(r)$. We denote by $Z(r)$ the center of $SU(r)$ which consists of multiples of the identity matrix by $r$-th roots of unity.

\begin{lem}
Let $A^{\dagger}$ be a unitary connection on $P^{\dagger}$ that induces the connection $A$ on $P$. Then the following are equivalent.
\begin{enumerate}[label=(\alph*)]
\item $\Stab(A) = Z(r)$. 
\item $\dim H^0(E_A(X)) = 0$.
\item $\Hol(A^{\dagger})$ acts irreducibly on $\C^r$. 
\end{enumerate}
\end{lem}

\begin{proof}
$(a) \Longrightarrow (b)$. This follows directly from the fact that $H^0(E_A(X))$ is the Lie algebra of $\Stab(A)$. 

$(b) \Longrightarrow (c)$. Suppose the action of $\Hol(A^{\dagger})$ splits into $\C^r$ two invariant subspaces $\C^{r_1}$ and $\C^{r_2}$. For each element $\zeta \in U(1)$, we define a map on $\C^r$ by multiplying the first $r_1$ factors by $\zeta^{r_2}$ and the remaining $r_2$ factors by $\zeta^{-r_1}$. Then each map defines an element in $SU(r)$ that intertwines with the action of $\Hol(A^{\dagger})$. Such an assignment embeds a neighborhood of $1 \in U(1)$ in $SU(r)$. This implies that $\dim H^0(E_A(X)) \geq 1$. 

$(c) \Longrightarrow (a)$. Since $\Hol(A^{\dagger})$ acts irreducibly on $\C^r$, Schur's Lemma tells us that each element in $\Stab(A)$ must act on $\C^r$ via a scalar multiplication. Since $\Stab(A)$ lies in $SU(r)$ and contains $Z(r)$, we conclude that $\Stab(A) = Z(r)$. 
\end{proof}

\begin{dfn}
A connection $A \in \A_k$ is said to be irreducible if its isotropy group $\Stab(A) = Z(r)$, and reducible otherwise. We denote the space of irreducible connections of class $L^2_k$ by $\A^*_k$. 
\end{dfn}

Let $A$ be a $G_2$-instanton. Since each element in $\Stab(A)$ is $A$-parallel, one readily checks that $E_A(X)$ is $\Stab(A)$-equivariant with respect to the adjoint action. The first cohomology $H^1(E_A(X))$ represents the Zariski tangent space of $[A]$ in $\M_{\phi}(X, P^{\dagger})$. The argument in \cite[Chapter 4]{DK90} can be applied to obtain a Kuranishi obstruction map $\mathfrak{o}_A: H^1(E_A(X)) \to H^2(E_A(X))$ so that a neighborhood of $[A]$ in $\M_{\phi}(X, P^{\dagger})$ can be identified with a neighborhood in $\mathfrak{o}_A^{-1}(0)/\Stab(A)$. Note that when $\Stab(A) = Z(r)$ is the center, the action of $\Stab(A)$ on $E_A(X)$ is trivial. In particular, when $H^0(E_A(X))=H^2(E_A(X)) = 0$, a neighborhood of $[A]$ is identified with a neighborhood of the origin in $H^1(E_A(X))$. This interpretation motivates the following definition.

\begin{dfn}\label{d3.7}
A $G_2$-instanton $[A] \in \M_{\phi}(X, P^{\dagger})$ is said to be regular if $H^2(E_A(X)) = 0$. 
\end{dfn}

By taking adjoints on the even part of $E_A(X)$, one can wrap-up the information of $E_A(X)$ into a single elliptic operator:
\begin{equation}
L_A:=
\begin{pmatrix}
0 & - d^*_A \\
-d_A & \star(d_A \wedge \psi)
\end{pmatrix}
: L^2_k(X, (\Lambda^0 \oplus \Lambda^1) \otimes \g_P) \longrightarrow L^2_{k-1}(X, (\Lambda^0 \oplus \Lambda^1) \otimes \g_P). 
\end{equation}
Note that $L_A$ is self-adjoint with respect to the $L^2$-inner product. So we conclude
\[
0=\ind L_A = h^1(E_A(X)) + h^3(E_A(X)) - h^0(E_A(X)) - h^2(E_A(X)),
\]
where $h^i(E_A(X)) = \dim H^i(E_A(X))$. From the discussion above, we conclude that a regular irreducible $G_2$-instanton is isolated in the moduli space. 

\subsection{\em Reducible $G_2$-Instantons} \label{ss3.3} \hfill

\vspace{3mm}

We start with a line bundle $L$ over a $G_2$-orbifold $(X, \phi)$. Denote by $\omega_{\phi}$ the harmonic $2$-form representing $c_1(L)$ with respect to the $G_2$-structure $\phi$. The $G_2$-representation $\Lambda^2 = \Lambda^2_7 \oplus \Lambda^2_{14}$ decomposes the space of harmonic $2$-forms into $\mathcal{H}^2(X) = \mathcal{H}^2_7(X) \oplus \mathcal{H}^2_{14}(X)$ since $\phi$ is parallel with respect to $g_{\phi}$. $G_2$-instantons over $L$ can be characterized by the following observation.

\begin{lem}
Let $L \to X$ be a line bundle. Then a unitary connection $A$ over $L$ is a $G_2$-instanton if and only if $F_A = -2\pi i \omega_{\phi}$ with $\omega_{\phi} \in \mathcal{H}^2_{14}(X)$. Moreover, gauge equivalence classes of $G_2$-instatnons are parametrized by the Jacobian torus 
\[
J(X):= H^1_{\orb}(X; i\R)/H^1_{\orb}(X; \Z).
\]
\end{lem}

\begin{proof}
Let $A$ be a $G_2$-instanton over $L$. Then (\ref{e2.5}) tells us that 
\[
\star( F_A \wedge \phi ) = -F_A  \Longrightarrow \star F_A = - F_A \wedge \phi. 
\]
The Bianchi identity then implies that $F_A$ is harmonic. Since $i/2\pi F_A$ represents $c_1(L)$, we conclude that $F_A = -2\pi i \omega_{\phi}$. Since $\pi_7(F_A) = 0$, we see that $\omega_{\phi} \in \mathcal{H}^2_{14}(X)$. The converse follows from the definition of $G_2$-instantons. 

Given two $G_2$-instantons $A_1$ and $A_2$ over $L$, their difference $a = A_1 - A_0 \in \Omega^1(X; i\R)$ is closed since their curvature forms coincide with $-2\pi i \omega_{\phi}$. Any unitary gauge transformation in the identity component of the gauge group takes the form $u = e^{\xi}$ with $\xi \in \Omega^0(X; i\R)$. Since $e^{\xi} \cdot A = A - d\xi$, we see that, up to gauge transformations in the identity component, $G_2$-instantons are identified with $H^1(X; i\R)$. Passing to the classifying space of the orbifold $X$, we see that the components of the unitary gauge group are identified with $H^1_{\orb}(X; \Z)$. The result then follows by identifying the action of elements from different components of the gauge group with the action of $H^1_{\orb}(X; \Z)$ on $H^1_{\orb}(X; \R)$ as in the case of manifolds. 
\end{proof}

With the knowledge of reducible $G_2$-instantons over line bundles, we can obtain a necessary condition for the existence of reducible projective $G_2$-instantons over unitary bundles of higher rank, which is analogous to that of the ASD instanton case \cite[Proposition 2.3]{K04}. 

\begin{lem}\label{l3.9}
Let $P^{\dagger}$ be a principal $U(r)$-bundle over a $G_2$-orbifold $(X, \phi)$. Denote by $\omega^{\dagger}_{\phi} \in \mathcal{H}^2(X)$ the harmonic representative of $c_1(\det P^{\dagger})$. Suppose $P^{\dagger}$ admits a reducible projective $G_2$-instanton. Then one can find an integer $r' \in (0, r)$ and an integral harmonic class $\omega \in \mathcal{H}^2(X)$ such that 
\[
\omega - \frac{r'}{r} \omega^{\dagger}_{\phi} \in \mathcal{H}^2_{14}(X). 
\]
\end{lem}

\begin{proof}
The argument is completely analogous to that of \cite[Proposition 2.3]{K04}. Let $A^{\dagger}$ be a reducible projective $G_2$-instanton over $P^{\dagger}$. Then $A^{\dagger}$ reduces the structure group of $P^{\dagger}$ to $U(r_1) \times U(r_2)$ for some integers $r_1$, $r_2$ satisfying $r_1+r_2 = r$, i.e. one can find a sub-bundle $Q^{\dagger} \subset P^{\dagger}$ of the form $Q^{\dagger} = P^{\dagger}_1 \times_X P^{\dagger}_2$, where $P^{\dagger}_i$ is a principal $U(r_i)$-bundle over $X$. The connection $A^{\dagger}$ decomposes accordingly as $A^{\dagger}_1 \oplus A^{\dagger}_2$. Denote by $\theta_i = \det A^{\dagger}_i$ the induced connection on the determinant line. Then the corresponding curvature is given by the trace $F_{\theta_i} = \tr F_{A^{\dagger}_i}.$ We also denote $\theta = \det A^{\dagger}$. Since $A^{\dagger}=A^{\dagger}_1 \oplus A^{\dagger}_2$ is a projective $G_2$-instanton, we conclude that 
\[
\left(F_{A^{\dagger}_i} - \frac{1}{r_i} \tr F_{A^{\dagger}_i} \otimes \id \right) \wedge \psi = 0. 
\] 
The same formula also holds for $A^{\dagger}$. From this formula, we conclude that 
\[
\pi_7 \left( \tr(F_{A^{\dagger}})/r \right) = \pi_7 \left( \tr(F_{A^{\dagger}_1})/r_1 \right) =  \pi_7 \left( \tr(F_{A^{\dagger}_2})/r_2\right). 
\]
In particular, we get 
\[
\left(F_{\theta_1} - \frac{r_1}{r} F_{\theta} \right) \wedge \psi = 0. 
\]
Now the result follows from the Chern--Weil homomorphism by taking $\omega$ to be the harmonic representative of $c_1(P^{\dagger}_1)$ and $r' = r_1$. 
\end{proof}

Usually \autoref{l3.9} is not very helpful with avoiding reducible $G_2$-instantons in contrast with the ASD instanton case. Because any torsion-free $G_2$-manifold $(M, \phi)$ whose holonomy group coincides with $G_2$ has $b_1(M) = 0$. In this case, the necessary condition imposed by \autoref{l3.9} automatically holds. However, when the holonomy group is a proper subgroup of $G_2$, one can indeed use \autoref{l3.9} to avoid reducibles. Since we primarily work with flat connections, the reducible ones can be avoided by purely imposing topological assumptions on $P^{\dagger}$. 

\begin{dfn}\label{d3.10}
A principal $U(r)$-bundle $P^{\dagger}$ over a $G_2$-orbifold $(X, \phi)$ is called admissible if 
\begin{enumerate}
\item $(r-1)c_1^2(P^{\dagger}) - 2r c_2(P^{\dagger}) = 0 \in H^4_{\orb}(X; \Z)$.
\item There exists $\lambda \in H^{\orb}_2(X; \Z)$ such that $\langle c_1(\det P^{\dagger}), \lambda \rangle$ is coprime to $r$. 
\end{enumerate}
\end{dfn}

\begin{cor}
Let $P^{\dagger}$ be an admissible principal $U(r)$-bundle over a $G_2$-orbifold $(X, \phi)$. Then $P^{\dagger}$ does not admit any reducible projective $G_2$-instantons.
\end{cor}

\begin{proof}
Let $A^{\dagger}$ be a projective $G_2$-instanton over $P^{\dagger}$, and $A$ the corresponding connection over the adjoint $PU(r)$-bundle $P$. (\ref{e3.3}) tells us that $\|F_A\|^2_{L^2} = 4\pi^2/r \langle p_1(\g_P) \smile [\phi], [X] \rangle$. Since $p_1(\g_P) = (r-1)c_1^2(P^{\dagger}) - 2rc_2(P^{\dagger})$, we conclude that $F_A=0$. Thus $A^{\dagger}$ is projectively flat. 

Suppose $A^{\dagger}$ decomposes as $A^{\dagger} = A^{\dagger}_1 \oplus A^{\dagger}_2$. The argument of \autoref{l3.9} implies that $F_{\theta_1} - r_1/r F_{\theta} = 0$. In particular, the Chern--Weil homomorphism implies that $r_1/r c_1(\det P^{\dagger})$ is an integral class. This contradicts to the fact that $\langle c_1(\det P^{\dagger}), \lambda \rangle$ is coprime to $r$ for any $\lambda \in H^{\orb}_2(X; \Z)$ since we can choose $\lambda \in H^{\orb}_2(X; \Z)$ so that the pairing between $\lambda$ and $r_1/r \cdot c_1(\det P^{\dagger})$ to be a non-zero integer. 
\end{proof}

\subsection{\em Orientation in Simple Cases} \label{ss3.4} \hfill

\vspace{3mm}

The orientability issue of the moduli space $\M_{\phi}(X, P^{\dagger})$ has been addressed by Donaldson--Segal \cite{DS11} and Walpuski \cite{W13} in the manifold case. Recently Joyce--Upmeier \cite{JU21} established a canonical assignment for orientations of the moduli space with an extra input called `flag structures'. Rather than pursuing a full generalization to $G_2$-orbifolds, we will work out the orientability for simple $G_2$-orbifolds. 

We write $\Ind L_A:= \ker L_A \otimes (\cok L_A)^*$, and $\Ind \mathbb{L} \to \A_k$ for the vector bundle whose fiber at $A$ is $\Ind L_A$. The gauge group $\G_{k+1}$ acts on $\Ind L_A$ by extending the action $u \omega u^{-1}$ on $\g_P$-valued forms  $\omega$. 

\begin{prop}\label{p3.13}
Let $(X, \phi)$ be a simple $G_2$-orbifold, and $P^{\dagger} \to X$ a $U(r)$-bundle with either $r=2$ or $c_1(P^{\dagger})$ even. Then $\G_{k+1}$ acts trivially on $\det \Ind \mathbb{L}$. 
\end{prop}

Note that \autoref{p3.13} implies that $\M_{\phi}(X, P^{\dagger})$ is orientable. Because the tangent space $T_{[A]}\M$ of a regular irreducible connection is identified with $H^1(E_A(X))$ whose determinant is identified with $\det E_A(X)$. One can further identify $\det E_A(X) \simeq \det \Ind L_A$ canonically using the metric on $\g_P$-valued forms. Since $\A_k$ is contractible, we know that $\det \Ind \mathbb{L}$ is trivial. The triviality of the $\G_{k+1}$-action tells us that the bundle $\det \Ind \mathbb{L}$ descends to a trivial bundle over $\A_k/\G_{k+1}$, which restricts to the orientation line bundle of $\M_{\phi}(X, P^{\dagger})$ in the regular case. 

\begin{cor}
Let $(X, \phi)$ be a simple $G_2$-orbifold, and $P^{\dagger} \to X$ a $U(r)$-bundle with either $r=2$ or $c_1(P^{\dagger})$ even. Suppose the irreducible moduli space $\M^*_{\phi}(X, P^{\dagger})$ is regular. Then $\M^*_{\phi}(X, P^{\dagger})$ is orientable. 
\end{cor}

The proof of \autoref{p3.13} is a standard application of the Atiyah--Singer's index theory. To set it up, we write $P_u$ for the bundle over $S^1 \times X$ obtained by identifying the boundary of $[0, 1] \times P$ via the gauge transformation $u \in \G_{k+1}$.  Given $A \in \A_k$, we choose a smooth path $(A_t)_{t \in [0,1]}$ with $A_0 = A$ and $A_1 = u\cdot A$. Then we get an elliptic operator on $\g_{P_u}$-valued forms of $S^1 \times X$:
\[
S_u(A) := \frac{d}{dt} + L_{A_t}: L^2_k(S^1 \times X, (\Lambda^0 \oplus \Lambda^1) \otimes \g_{P_u}) \longrightarrow L^2_{k-1}(S^1 \times X, (\Lambda^0 \oplus \Lambda^1) \otimes \g_{P_u}).
\]
The index can be related to the spectral flow $\ind S_u(A) = \Sf(L_{A_t})$, which is independent of the choice of the paths $(A_t)$ from the standard theory of Atiyah--Patodi--Singer \cite{APS3}. Moreover, the weight of the $u$-action on $\det \Ind \mathbb{L}$ is given by $(-1)^{\Sf(L_{A_t})}$. Thus the proof of \autoref{p3.13} is reduced to the following statement.

\begin{lem}\label{l3.15}
Under the assumption of \autoref{p3.13}, we know $\ind S_u(A)$ is even for all $A \in \A_k$ and $u \in \G_{k+1}$. 
\end{lem}

\begin{proof}
The $G_2$-structure $\phi$ determines a spin structure $\s$ on $S^1 \times X$. As explained in \cite{W13}, the principal symbol of $S_u(A)$ is dual to that of the Dirac operator $D_{\mathbb{A}}$ on $\s$ twisted by $\g_{P_u}$, where $\mathbb{A}= d/dt + A_t$ is the induced connection on $\g_{P_u}$. To compute the index of $D_{\mathbb{A}}(\g_{P_u})$, we make use of the orbifold generalization of Atiyah--Singer's index theorem in \cite{K81}. 

Let's write $Y = S^1 \times X$ temporarily. For each $y \in Y$, the isotropy group $G_y = \Z/2$ or $V_4$, where $V_4$ is the Klein four-group. Explicitly we write $\Z/2 = \{1, -1\}$ and $V_4=\{1, a, b, ab\}$ with $a = (-1, 1)$ and $b=(1, -1)$. We consider the modified singular sets:
\[
\Sigma_I:= \{ (y, -1) \in Y \times \Z/2 : G_y = \Z/2\}, \quad \Sigma_{V} : = \{(y, g) \in Y \times V_4: G_y = V_4, \;\; g \neq 1\}. 
\]
These singular sets patch together to an ineffective orbifold $\Sigma_Y:= \Sigma_I \sqcup_{\sim} \Sigma_V$, where $(y, -1) \sim (y, g)$ if and only if there is an inclusion of orbifold charts $(\tilde{U}_y, G_y) \hookrightarrow (\tilde{U}_{y'}, G_{y'})$ that sends $-1 \in G_y$ to $g \in G_{y'}$. To see the orbifold structure on $\Sigma_Y$, we write $\tilde{U}^g_y$ for the $g$-invariant subset of $\tilde{U}_y$. Then the orbifold atlas over $\Sigma_Y$ is given by $\{(\tilde{U}^g_y, G_y) : y \in Y, g \neq 1 \in G_y\}$. Note that $\tilde{U}^g_y$ is an open ball of dimension either $2$ or $4$ when $g$ is non-trivial. Thus we can write $\Sigma_Y = \Sigma_Y^{(2)} \cup \Sigma_Y^{(4)}$ into the disjoint of $2$- and $4$-dimensional components. We also need to consider the normal bundle $N(\Sigma_Y)$ of $\Sigma_Y$ locally given by $(\tilde{U}_y/\tilde{U}^g_y, G_y) \to (\tilde{U}^g_y, G_y)$. In particular, $g$ acts on the fiber $N(\Sigma_Y)|_{(y, g)}$ by $-1$-multiplication. 

With the above set-up, the orbifold version of Atiyah--Singer's index theorem \cite{K81} tells us that  
\[
\begin{split}
\ind D_{\mathbb{A}}(\g_{P_u}) & = \int_Y \ch(\g^{\C}_{P_u}) \hat{A}(Y) + \frac{1}{2}\int_{\Sigma^{(4)}_Y} \ch(\g^{\C}_{P_u}) \hat{A}_{\pi}(N(\Sigma^{(4)}_Y))\hat{A}(\Sigma^{(4)}_Y) \\
& + \frac{1}{4}\int_{\Sigma^{(2)}_Y}\ch(\g^{\C}_{P_u}) \hat{A}_{\pi}(N(\Sigma^{(2)}_Y))\hat{A}(\Sigma^{(2)}_Y)\\
&=: I_1(\g^{\C}_{P_u}) + I_2(\g^{\C}_{P_u}) + I_3(\g^{\C}_{P_u}), 
\end{split}
\]
where $\g^{\C}_{P_u} = \g_{P_u} \otimes \C$, and $\hat{A}_{\pi}$ is the twisted $\hat{A}$-class in \cite[Page 267]{LM89}. 

The Chern character of $\g^{\C}_{P_u}$ can be computed as follows. The differential of the adjoint homomorphism gives us an isomorphism $\ad: \su(r) \to \mathfrak{pu}(r)$. Thus we can identify $\g_{P_u} = \su(P^{\dagger}_u)$. Furthermore, $\End(P^{\dagger}_u) = \su(P^{\dagger}_u) \otimes \C \oplus \underline{\C}$. After throwing away the terms of degree greater than $8$, we can compute
\[
\begin{split}
\ch(P^{\dagger}_u) & = r+s_1+\frac{1}{2}s_2+\frac{1}{6}s_3+\frac{1}{24}s_4\\
& = r + c_1 + \frac{1}{2}(c_1^2 - 2c_2) + \frac{1}{6}(c_1^3 - 3c_2c_1 +3c_3) \\
& + \frac{1}{24}(c_1^4 - 4c_2c_1^2 + 4c_3c_1 + 2c_2^2 - 4c_4),
\end{split}
\]
where $s_i$ is the polynomial that expresses the ith power sum into elementary symmetric polynomials whose inputs are the Chern classes $c_i$ of $P^{\dagger}_u$. Thus the Chern character of $\g^{\C}_{P_u}$ is given by 
\[
\begin{split}
\ch(\g^{\C}_{P_u}) & = \ch(\End(P^{\dagger}_u)) - 1 = \ch(P^{\dagger}_u)\ch(P^{\dagger, *}_u) - 1 \\
& = \left(r+s_1+\frac{1}{2}s_2+\frac{1}{6}s_3+\frac{1}{24}s_4 \right)\left(r-s_1+\frac{1}{2}s_2-\frac{1}{6}s_3+\frac{1}{24}s_4 \right) - 1 \\
& = r^2 - 1 + rs_2 - s^2_1 + \frac{rs_4}{12} - \frac{s_1s_3}{3} + \frac{s^2_2}{4} \\
& = r^2 - 1 +(r-1)c_1^2 -2rc_2 \\
& + \frac{1}{12} \left((r-1)c_1^4 - 4rc_2c_1^2 + (4r-12)c_3c_1 + (2r+12)c^2_2-4rc_4 \right).
\end{split}
\]

Note that $Y=S^1 \times X$ is a product with circle. So up to the eighth degree, the $\hat{A}$-class is given by $\hat{A} = 1 -p_1/24$, where $p_1:=p_1(Y)$. Note that the normal bundle $N(\Sigma_Y)$ is pulled back from the normal bundle of the correspondingly constructed singular sets $\Sigma_X$ whose components have dimension either $1$ or $3$. Up to the third degree, $\hat{A}_{\pi} = \pm 1$, where the sign can be worked out following \cite{AB68}. 

Since $\det u = \id$, we know that 
\[
c_1(P^{\dagger}_u)  =c_1(\det P^{\dagger}_u)  = \pi^* c_1(\det P^{\dagger}) = \pi^*c_1(P^{\dagger}),
\]
where $\pi: S^1 \times X \to X$ is the projection. Thus by dimension counting, we know 
\[
\int_Y c_1^4 = 0 \qquad \int_Y c_1^2 \cdot p_1 = 0 \qquad \int_{\Sigma^{(4)}_Y} c_1^2 = 0 \qquad \int_{\Sigma^{(2)}_Y} c_1 = 0. 
\] 

Now we can compute 
\[
\begin{split}
& I_1(\g^{\C}_{P_u})  = \frac{1}{12} \int_Y rc_2 \cdot p_1 + \frac{1}{6}\int_Y -2rc_2c_1^2 + (2r-6)c_3c_1 + (r+6)c^2_2 - 2rc_4   \\
& I_2(\g^{\C}_{P_u}) + I_3(\g^{\C}_{P_u})  = \pm \int_{\Sigma^{(4)}_Y} rc_2. 
\end{split}
\]
We can twist the spin bundle of $Y$ by $P^{\dagger}$, and get corresponding terms:
\[
\begin{split}
& I_1(P^{\dagger}) =\frac{1}{24}\int_Y c_2 \cdot p_1 + \frac{1}{12} \int_Y -2c_2c_1^2 + 2c_3c_1 + c^2_2 - 2c_4 \\
& I_2(P^{\dagger}) + I_3(P^{\dagger}) = \pm \frac{1}{2} \int_{\Sigma^{(4)}_Y} c_2. 
\end{split}
\]
The above computation tells us that 
\[
\ind D_{\mathbb{A}}(\g_{P_u}) - (12 + 2r) \ind D_{\mathbb{A}}(P^{\dagger}) \equiv -\frac{1}{2}\int_Y c_2 \cdot p_1 - 3\int_Y c_3 \cdot c_1 \; \mod 2. 
\]
Since $Y$ is spin, we know $p_1(Y)$ is divisible by four. Thus we conclude that 
\[
\ind D_{\mathbb{A}}(\g_{P_u}) \equiv  \int_Y c_3(P^{\dagger}_u) \cdot \pi^*c_1(P^{\dagger}) \; \mod 2
\]
So when $\rk(P^{\dagger}) = 2$ or $c_1(P^{\dagger})$ is even, we know that $\ind D_{\mathbb{A}}(\g_{P_u})$ is even.  
\end{proof}

\begin{rem}
The last product term $c_3c_1$ also appears in the manifold case if one considers $U(r)$-bundles. Joyce--Upmeier \cite[Theorem 1.2]{JU21} proved the orientability for $U(r)$-bundles with the help of the embedding $U(r) \hookrightarrow SU(r+1)$. Thus $c_3c_1$ has to be even at least in the manifold case, although we don't know a good reason for this.
\end{rem}

\section{\large \bf Perturbations and Transversality}\label{s4}

Holonomy perturbations have been constructed for achieving transversality of ASD connections in dimension 4 \cite{D87} and flat connections in dimension 3 \cite{F88}. In this section, we argue that the transversality of the irreducible stratum of $G_2$-instantons can also be achieved via holonomy perturbations.

\subsection{\em Holonomy Perturbations} \label{ss4.1} \hfill

\vspace{3mm}

Let $P^{\dagger}$ be a $U(r)$-bundle over a $G_2$-orbifold $(X, \phi)$. Suppose we have an $l$-tuple $\pmb{\rho}=(\rho_1, ..., \rho_l)$ of smooth immersions $\rho_i: S^1 \times D^6 \to X$ satisfying 
\begin{equation}\label{e4.1}
\rho_i(s, x) = \rho_j(s, x), \quad \text{ for all } s \in [-\epsilon, \epsilon], \; x \in D^6
\end{equation}
where $s$ is the parameter of $S^1$ representing $e^{i 2\pi s}$, and $\epsilon > 0$ is a small constant. Denote by $\rho_x: S^1 \to X$ the orbi-loop given by $\rho(-, x)$ based at $\rho(0, x)$. Let $\eta: U(r)^l \to \R$ be a smooth function that is invariant under the diagonally $SU(2)$-adjoint action, and $\nu \in \Omega^2(D^6)$ be a non-negative $2$-form supported near the center of $D^6$ so that $\int_{D^6} \nu = 1$. 

\begin{dfn}
Given $(\pmb{\rho}, \eta, \nu)$ as above, we define
\[
\begin{split}
\tau_{\pmb{\rho}, \eta}: \A_k & \longrightarrow \R \\
A & \longmapsto \int_{D^6} \eta \left(\Hol_{\rho_{1, x}}(A^{\dagger}), ..., \Hol_{\rho_{l, x}}(A^{\dagger})\right) \nu
\end{split}
\]
to be the associated cylinder function over $\A_k$. 
\end{dfn}

Now we derive the formal gradient of the function $\tau_{\pmb{\rho}, \eta}$. Denote the $i$-the partial derivative of the function $\eta$ by $\partial_i \eta: U(r)^l \to \mathfrak{u}(r)^*$, which arises as the restriction of $d\eta$ to the $i$-th factor $T^*U(r) \simeq \mathfrak{u}^*(r)$ of $T^*U(r)^l$. Due to the requirement that $\rho_i(0, x) = \rho_j(0, x)$, there is no ambiguity in writing $\Hol_{\pmb{\rho}, x}(A^{\dagger}) := \left(\Hol_{\rho_{1, x}}(A^{\dagger}), ..., \Hol_{\rho_{l, x}}(A^{\dagger}) \right)$. Then the differential of the cylinder function $\tau_{\pmb{\rho}, \eta}$ takes the form 
\begin{equation}\label{e4.2}
d\tau_{\pmb{\rho}, \eta}|_A (a) = \sum_{i=1}^l \int_{D^6} \partial_i \eta|_{\Hol_{\pmb{\rho}, x}(A^{\dagger})}\left(-\int_0^1 (\rho^*_i a)(s) ds \cdot \Hol_{\rho_i, x} (A^{\dagger}) \right) \nu, \quad a  \in \Omega^1(X, \g_P). 
\end{equation}
Note that $\mathfrak{u}(r) = \mathfrak{su}(r) \oplus i\R$. We can compose $\partial_i \eta$ with the restriction of $\mathfrak{u}(r)^*$ to $\mathfrak{su}(r)^*$ and identify $\mathfrak{su}(r)^*$ with $\mathfrak{su}(r)$ via the Killing form. The resulting operator is denoted by $(\partial_i \eta)^0_{\dagger}: U(r)^l \to \mathfrak{su}(r)$. The conjugation invariance of $\eta$ implies that $(\partial_i \eta)^0_{\dagger}(\Hol_{\pmb{\rho}}(A^{\dagger}))$ is a section of $\rho_i^*\g_P$ over $\{0 \} \times D^6$. This section can be extended to a section over $S^1 \times D^6$ via parallel transport of $A^{\dagger}$ along the $S^1$-factor, which might gain discontinuity over $\{ 0 \} \times D^6$ with jumps given by $\Hol_{\rho_i, x}(A^{\dagger})$-conjugation. Let's temporarily denote by $H_i$ this extended section of $\rho_i^*\g_P$ over $S^1 \times D^6$ (following the notation of \cite[Page 876]{KM11}). Then it follows from (\ref{e4.2}) that the gradient of $\tau_{\pmb{\rho}, \eta}$ is 
\begin{equation}\label{e4.3}
\grad \tau_{\pmb{\rho}, \eta} = \star_{\phi} \left(\sum_{i=1}^l (\rho_i)_* (H_i \otimes \nu) \right). 
\end{equation}
Although each single term $(\rho_i)_* (H_i \otimes \nu)$ might be discontinuous at $\{0\} \times D^6$, the diagonal conjugation invariance of $\eta$ implies that the sum on the RHS of (\ref{e4.3}) is at least continuous. We can actually get stronger estimates on the gradient of cylinder functions. 

\begin{lem}\label{l4.2}
Let $\tau$ be a cylinder function as above, and $a = A^{\dagger} - A^{\dagger}_0 \in L^2_k(X, \Lambda^1 \otimes \su(2))$ be the  connection form of a general connection $A^{\dagger} \in \A_k$. Then we have the following estimates.
\begin{enumerate}
\item There exists a constant $K_0$ so that 
\[
\|\grad \tau|_{A^{\dagger}}\|_{C^0} \leq K_0. 
\]
\item For each $1 \leq j \leq k$, there exists a constant $K_j$ so that 
\[
\|\grad \tau|_{A^{\dagger}}\|_{L^2_j} \leq K_j\left(1+ \|a\|^j_{L^2_j} \right). 
\]
\item For each positive integer $n \geq 1$ and $0 \leq j \leq k$, there exists a constant $K_{n,j}$ so that 
\[
\|D^n \grad\tau|_{A^{\dagger}}(a_1, ..., a_n)\|_{L^2_j} \leq K_{n, j}\left(1+\|a\|^j_{L^2_j} \right)\prod_{i=1}^n \|a_i\|_{L^2_j}.
\]
\end{enumerate}
All constants appeared above are independent of the connection $A^{\dagger}$. 
\end{lem}

\begin{proof}
For simplicity, we consider the case when there is a single smooth immersion $\rho: S^1 \times D^6 \to X$. The argument for general cases is almost identical except for extra complexity in notations. Now we assume the cylinder function takes the form $\tau (A) = \int_{D^6} \eta(\Hol_{\rho, x}(A^{\dagger})) \nu$ so that 
\[
\grad \tau = \star_{\phi} \rho_*(H \otimes \nu). 
\]
Let's write $A^c=\rho^*A^{\dagger}$ and $A^c_0 = \rho^*A^{\dagger}_0$ for the pulled-back connections over $S^1 \times D^6$, and $a^c=A^c - A^c_0$ the corresponding connection form. We can further decompose $a^c = b(s) + \beta(s)ds$, where $b(s)$ is an $S^1$-family of $\mathfrak{su}(2)$-valued $1$-forms over $D^6$, and $\beta(s)$ is an $S^1$-family of $\mathfrak{su}(2)$-valued $0$-forms over $D^6$. 

(1) Since the Lie group $U(r)$ is compact, we know $(d\eta)^0_{\dagger}(\Hol_{\rho}(A^{\dagger}))$ is bounded in $C^0$ by a constant that is independent of $A^{\dagger}$. Note that the section $H$ is obtained from $(d\eta)^0_{\dagger}(\Hol_{\rho}(A^{\dagger}))$ by parallel transporting along the $S^1$-direction of $S^1 \times D^6$. Since the space of $U(r)$-connections over $S^1$ is compact modulo gauge transformations, the gauge-invariant section $(d\eta)^0_{\dagger}(\Hol_{\rho}(A^{\dagger}))$ gives rise a uniform $C^0$-bound on $H$. 

(2) The constants in the statement will all depend on the reference connection $A^{\dagger}_0$. So we may assume $\rho^*A^{\dagger}_0$ is the trivial connection over $S^1 \times D^6$ to simplify notations. Then we know $\Hol_{\rho}(A^{\dagger}) = \exp(-\int^1_0 \beta(s) ds)$. From the construction of $H$ and gauge-invariance of $\eta$, we know 
\begin{equation}\label{e4.4.1}
\begin{split}
H|_{(s, x)} & =\exp(-\int^s_0 \beta(t) dt) \cdot (d\eta)^0_{\dagger}(\Hol_{\rho, x}(A^{\dagger})) \cdot \exp(\int^s_0 \beta(t) dt) \\
& =  (d\eta)^0_{\dagger}\left(\exp(-\int^s_0 \beta(t) dt) \cdot \exp(-\int^1_0 \beta(s) ds) \cdot \exp(\int^1_0 \beta(s) ds) \right).
\end{split}
\end{equation}
So the derivatives of $H$ along the $S^1$-direction involves the composition of derivatives of $(d\eta)^0_{\dagger}$ with products among the terms
\begin{equation}\label{e4.4}
\partial^i_s \beta(s), \quad \exp(-\int^s_0 \beta(t) dt), \quad \exp(\int^1_0 \beta(s) ds),
\end{equation}
where $\partial^i_s$ means to take the $i$-th derivative along the $S^1$-direction. 
The derivatives of $H$ along the $D^6$-direction involves the composition of derivatives of $(d\eta)^0_{\dagger}$ with products among the terms
\begin{equation}\label{e4.5}
\int_0^s \nabla^i_x \beta(t)dt, \quad \int_0^1 \nabla^k_x \beta(t)dt, \quad \exp(-\int^s_0 \beta(t) dt), \quad \exp(\int^1_0 \beta(s) ds),
\end{equation}
where $\nabla^i_x$ means to take the $i$-th derivative along the $D^6$-direction. Recall that $a^c = \rho^*A^{\dagger} - \rho^*A^{\dagger}_0$ is the connection form on $S^1 \times D^6$. It's clear that 
\[
\|\partial^i_s \beta(s)\|_{L^2(S^1 \times D^6)} \leq const. \|a^c\|_{L^2_i(S^1 \times D^6)}
\]
As for the $L^2$-norm of $\int^s_0 \nabla^i_x \beta(t)dt$, we first apply integration by parts to get the following estimates:
\begin{equation}\label{e4.7}
\begin{split}
\int_0^1 \left|\int^s_0\nabla^i_x \beta(t) dt\right|^2 ds & \leq \int_0^1 s \int^s_0 |\nabla^i_x \beta(t)|^2 dt ds \\
& \leq \frac{1}{2} \left( \int^1_0 |\nabla^i_x \beta(t)|^2 dt - \int^1_0 s  |\nabla^i_x \beta(s)|^2 ds \right) \\
& \leq \int^1_0 |\nabla^i_x \beta(t)|^2 dt. 
\end{split}
\end{equation}
Thus we conclude that 
\begin{equation}
\left\| \int^s_0 \nabla^i_x \beta(t)dt\right\|_{L^2(S^1 \times D^6)} \leq \|\nabla^u_x\beta(s)\|_{L^2(S^1 \times D^6)} \leq const. \|a^c\|_{L^2_i(S^1 \times D^6)}. 
\end{equation}
Since the $j$-th derivative of $H$ is given by the composition of derivatives of $(d\eta)^0_{\dagger}$ with products of terms in (\ref{e4.4}) and (\ref{e4.5}) so that in each product the total order of derivatives add up to $j$, we conclude that 
\[
\|H\|_{L^2(S^1 \times D^6)} \leq const. \left(1 + \|a^c\|^j_{L^2_j(S^1 \times D^6)} \right),
\] 
which further implies the claimed estimate 
\[
\|\star_{\phi}\rho_*(H\otimes \nu)\|_{L^2_j} \leq const. \left(1 + \|a\|^j_{L^2_j} \right). 
\]

(3) Once we know the expression of $H$ in (\ref{e4.4.1}), the proof is almost the same as that of \cite[Proposition 3.2]{M20}. The only difference comes from the adjoint action of $\exp(\int^s_0\beta(t)dt)$, whereas in \cite[Proposition 3.2]{M20} one does not need to apply further parallel transport. This extra complication can dealt with as above using the estimate (\ref{e4.7}). 
\end{proof}

\autoref{l4.2} tells us that the gradient of a cylinder function defines a map 
\[
\grad \tau: \A_k \longrightarrow L^2_k(X, \Lambda^1 \otimes \g_P).
\]
Furthermore, this map is smooth due to (3) of \autoref{l4.2}. Now we can proceed to construct a Banach space of perturbations. 

\begin{dfn}\label{d4.3}
Fix a countable family of cylinder functions $\pmb{\tau} = \{\tau\}_{\alpha \in \mathbb{N}}$ satisfying the following assumptions.
\begin{enumerate}
\item For each integer $l > 0$, one can find a sub-family of $l$-tuples of smooth immersions $\{\pmb{\rho}_{\alpha}\}$ associated to $\{\tau_{\alpha}\}$ that are dense in the space of all $l$-tuples of smooth immersions satisfying (\ref{e4.1}) with respect to $C^1$-topology. 

\item For each integer $l > 0$, one can find a sub-family of smooth $SU(2)$-invariant functions $\eta_{\alpha}: U(r)^l \to \R$ associated to $\tau_{\alpha}$ that are dense in the space of all smooth $SU(2)$-invariant functions on $U(r)^l$ with respect to $C^{\infty}$-topology. 
\end{enumerate}
Then we fix $C_{\alpha}:= \sup \{K_{i, \alpha}: 0 \leq i \leq \alpha\}$, wehre $K_{i, \alpha}$ is the constant appeared in \autoref{l4.2} associated to the cylinder function $\tau_{\alpha}$. We define
\[
\mathscr{P}:= \left\{\pmb{\pi} = \{\pi_{\alpha}\}_{\alpha \in \mathbb{N}}: \pi_{\alpha} \in \R, \; \; \sum_{\alpha} C_{\alpha} |\pi_{\alpha}| < \infty \right\}
\]
to be the space of sequences $\pmb{\pi} = \{\pi_{\alpha}\}$ of finite norm $\|\pmb{\pi}\|:= \sum_{\alpha} C_{\alpha} |\pi_{\alpha}|$. We refer to $\mathscr{P}$ as the space of holonomy perturbations. 
\end{dfn}

Due to the construction of $\mathscr{P}$, the sum $\tau_{\pmb{\pi}}:= \sum_{\alpha} \pi_{\alpha}\tau_{\alpha}$ defines a smooth function on $\A_k$, which we use to perturb the Chern--Simons functional $\cs_{\phi}$:
\[
\cs_{\phi, \pmb{\pi}}:= \cs_{\phi} + \tau_{\pmb{\pi}}: \A_k \to \R. 
\]
Then we define the perturbed $G_2$-instanton equation to be the vanishing equation of the gradient of the perturbed functional: 
\begin{equation}\label{e4.9}
\grad \cs_{\phi, \pmb{\pi}}(A) = 0 \Longleftrightarrow F_A \wedge \psi + \sigma_{\pmb{\pi}}(A) = 0,
\end{equation}
where $\sigma_{\pmb{\pi}}(A) = \sum_{\alpha} \pi_{\alpha} \left(\sum_{i=1}^{l_{\alpha}} (\rho_{\alpha, i})_* (H_{\alpha, i} \otimes \nu) \right)$ with $\pmb{\rho}_{\alpha} = (\rho_{\alpha, 1}, ..., \rho_{\alpha, l_{\alpha}})$.

\subsection{\em Transversality} \label{ss4.2} \hfill

\vspace{3mm}

In this subsection, we establish the transversality of the moduli space of perturbed irreducible $G_2$-instantons. Following the notations in \cite{KM07}, we write $\mathcal{T}_j$ for the tangent bundle of $\A_k$ completed with respect to the $L^2_j$-norm for $j \leq k$, which is identified with $\A_k \times L^2_j(X, \Lambda^1 \otimes \g_P)$. The infinitesimal action of the gauge group $\G_{k+1}$ decomposes the tangent bundle into two pieces:
\begin{equation}
\mathcal{T}_j := \mathcal{J}_j \oplus \mathcal{K}_j,
\end{equation}
where $\mathcal{J}_j|_A = \im d_A$ is the $L^2_j$-tangent space of the gauge orbit at $A$, and $\mathcal{K}_j|_A = \ker d^*_A$ is the $L^2$-orthogonal complement. The local gauge invariant of $\cs_{\phi, \pmb{\pi}}$ implies that its gradient $\grad \cs_{\phi, \pmb{\pi}}$ is a section of $\mathcal{K}_{k-1}$. One can also argue directly by writing $\grad \cs_{\phi, \pmb{\pi}}(A) = \star_{\phi} (F_A \wedge \psi + \sigma_{\pmb{\pi}}(A))$ and compute that $d_A \sigma_{\pmb{\pi}}(A) = 0$.

We consider the smooth map 
\begin{equation}\label{e4.10}
\begin{split}
\mathcal{F}: \mathscr{P} \times \A^*_k & \longrightarrow \mathcal{K}_{k-1} \\
(\pmb{\pi}, A) & \longmapsto \grad \cs_{\phi, \pmb{\pi}}(A) 
\end{split}
\end{equation}
The essence of the transversality problem is to show that the map $\mathcal{F}$  is transverse to the zero section of $\mathcal{K}_{k-1}$. To do so, we introduce the Hessian 
\begin{equation}
\begin{split}
\Hess_{A, \pmb{\pi}}: \mathcal{K}_k & \longrightarrow \mathcal{K}_{k-1} \\
a & \longmapsto \Pi_{\mathcal{K}} (D \grad \cs_{\phi, \pmb{\pi}}|_A a),
\end{split}
\end{equation}
where $\Pi_{\mathcal{K}}: \mathcal{T}_{k-1} \to \mathcal{K}_{k-1}$ is the fiberwise $L^2$-orthogonal projection. When $A$ is a critical point of $\cs_{\phi, \pmb{\pi}}$, one sees that $\Hess_{A, \pmb{\pi}}$ is a self-adjoint Fredholm map from the argument in \cite[Proposition 12.3.1]{KM07}. 

\begin{lem}\label{l4.4}
The map $\mathcal{F}$ defined in \eqref{e4.10} is transverse to the zero section of $\mathcal{K}_{k-1}$. 
\end{lem}

\begin{proof}
Let $(\pmb{\pi}, A) \in\mathcal{F}^{-1}(0)$ be an element in the zero set. The derivative of $\mathcal{F}$ at $(\pmb{\pi}, A)$ is 
\begin{equation}
D\mathcal{F}|_{(\pmb{\pi}, A)} (\pmb{\omicron}, a) = \Hess_{A, \pmb{\pi}} a +\star_{\phi} \sigma_{\pmb{\omicron}}(A),
\end{equation}
where $(\pmb{\omicron}, a) \in \mathscr{P} \times \mathcal{K}_k$. Since $\Hess_{A, \pmb{\pi}}$ is self-adjoint with respect to the $L^2$-inner product, we can identify the cokernel of $\Hess_{A, \pmb{\pi}}$ with $\ker \Hess_{A, \pmb{\pi}}$. Then it suffices to show that for each non-zero $a  \in \ker \Hess_{A, \pmb{\pi}}$, one can find $\pmb{\omicron} \in \mathscr{P}$ such that $\langle a, \star_{\phi} \sigma_{\pmb{\omicron}}(A) \rangle_{L^2} \neq 0$. 

Note that $\star_{\phi} \sigma_{\pmb{\omicron}}(A) = \grad \tau_{\pmb{\omicron}}$. We see that $\langle a, \star_{\phi} \sigma_{\pmb{\omicron}}(A) \rangle_{L^2}  = D\tau_{\pmb{\omicron}}|_A  (a)$. Suppose, on the contrary, there is an element $a  \in \ker \Hess_{A, \pmb{\pi}}$ such that $D\tau_{\pmb{\omicron}}|_A  (a) = 0$ for all $\pmb{\omicron}  \in \mathscr{P}$. Then for $A_{\epsilon} = A  + \epsilon a$ with $\epsilon  > 0$ sufficiently small, we have 
\[
\tau_{\pmb{\omicron}}(A)  = \tau_{\pmb{\omicron}}(A_{\epsilon}) \text{ and }  [A] \neq [A_{\epsilon}] \in \B^*_k = \A^*_k /\G_{k+1}.
\]
However, this is impossible for the following reason. Recall each term in $\tau_{\pmb{\omicron}}(A)$ is of the form $\int_{D^6} \eta_{\alpha}(\Hol_{\pmb{\rho}_{\alpha}}(A^{\dagger})) \nu$. Since the functions $\eta_{\alpha}$'s are chosen to be dense in the space of all $SU(2)$-invariant functions as in (2) of \autoref{d4.3}, we know for all words $W$ of a fixed length sufficiently large, say $L$, one has 
\[
\tr\left(W\left(\prod_{i=1}^{L}  \Hol_{\rho_i}(A^{\dagger})\right)\right) = \tr\left( W\left(\prod_{i=1}^{L}  \Hol_{\rho_i}(A_{\epsilon}^{\dagger})\right)\right)
\]
for all $L$-tuple immersions $\pmb{\rho} = (\rho_1, ..., \rho_L)$. Then \cite[Lemma 5.13]{D02} provides us with $g \in U(r)$ such that $g \Hol_{\rho_i}(A^{\dagger}) g^{-1}  =  \Hol_{\rho_i}(A_{\epsilon}^{\dagger})$. Since $A^{\dagger}$ and $A^{\dagger}_{\epsilon}$ has the same determinant, we know $g \in SU(r)$. Recall (1) in \autoref{d4.3} that the immersions $\rho_i$ are dense in the space of smooth immersions. Since both $A$ and $A_{\epsilon}$ are irreducible, we conclude that $A$ and $A_{\epsilon}$ are gauge equivalent, which is a contradiction. 
\end{proof}

\autoref{l4.4} has the following immediate consequence using the Sard--Smale theorem and the standard transversality argument in differential topology (c.f. \cite[Lemma 12.5.1]{KM07}). 

\begin{cor}
With respect to a generic perturbation $\pmb{\pi} \in \mathscr{P}$, every irreducible $\pmb{\pi}$-perturbed $G_2$-instanton is regular in the sense of \autoref{d3.7}. 
\end{cor}

\section{\large \bf Compactness of Moduli Spaces}\label{s5}

Due to the work of Uhlenbeck \cite{U82}, the moduli space of $G_2$-instantons is known to be non-compact for general bundles. In contrast to ASD connections in dimension $4$, a compactification of such moduli spaces is still far from commencement due to the wild behavior of the bubbling off subsets. Nevertheless, we are able to obtain the compactness of the moduli space of $G_2$-instantons over projectively flat bundles.

\begin{thm}\label{t5.1}
Let $P^{\dagger}$ be a projectively flat principal $U(n)$-bundle over a $G_2$-orbifold $(X, \phi)$. There exists a constant $\epsilon_0 > 0$ such that the moduli space $\M_{\phi,\pmb{\pi}}(X, P^{\dagger})$ of perturbed $G_2$-instantons is compact with respect to all $\pmb{\pi} \in \mathscr{P}$ satisfying $\|\pmb{\pi}\| < \epsilon_0$. 
\end{thm}

Since later in the proof we will use letter `$r$' for the radius of small balls, we choose to write $U(n)$-bundles rather than $U(r)$-bundles in this section to avoid confusion. The proof of \autoref{t5.1} is more or less standard in higher dimensional gauge theory (c.f. \cite{N88}). The argument adopted here was communicated to the author by Simon Donaldson in the case of Hermitian--Yang--Mills connections. Some parts of the proof are also sketched in \cite{D22}. We start by recalling the gauge fixing result in \cite{U82}. 

\begin{thm}[{\cite[Theorem 1.3]{U82}}]\label{t5.2}
Let $\mathbb{B} \subset \R^7$ be the unit ball and $p > 1$ an exponent. There exist constants $\epsilon_1 > 0$, $\mathfrak{c}_1 > 0$ such that each unitary connection $A$ on the trivial $SU(n)$-bundle satisfying $\|F_A\|_{L^{7/2}(\mathbb{B})} \leq \epsilon_1$ is gauge equivalent to a connection $A' = d+a'$ satisfying 
\begin{equation*}
(1) \; d^*a' = 0  \qquad
(2) \; \star a'|_{\partial \mathbb{B}} = 0 \qquad
(3) \; \|a'\|_{L^p_1(\mathbb{B})} \leq \mathfrak{c}_1 \|F_A\|_{L^p(\mathbb{B})}
\end{equation*}
\end{thm}

Recall $P$ is the induced $PU(r)$-bundle on which we have fixed a reference connection $A_0$. Let $A$ be a perturbed $G_2$-instanton which satisfies
\[
F_A \wedge \psi + \sigma_{\pmb{\pi}}(A) = 0
\]
for some $\pmb{\pi} \in \mathscr{P}$. We thus get 
\[
\pi_7(F_A) = - \star( \psi \wedge \star \sigma_{\pmb{\pi}}(A)). 
\]
From \eqref{e2.7} and the admissibility of $P$, we know that $\|F_A\|^2_{L^2} = 3 \|\pi_7(F_A) \|^2_{L^2}.$ With the $C^0$ bound on the perturbation, we get a priori bounds on $\|F_A\|_{L^2}$ and $\|\pi_7(F_A)\|_{C^0}$. In order to apply Uhlenbeck's result \autoref{t5.2}, we need to improve these bounds to the $L^{7/2}$ norm. To do this, we work locally over balls in $X$. 

\begin{lem}\label{l5.3}
Let $\mathbb{B} \subset \R^7$ be the unit ball equipped with the standard $G_2$-structure $\phi_0$, $B' \subset \mathbb{B}$ an interior ball, and $q \in (0, 7)$. There are constants $\epsilon_2 > 0$, $\mathfrak{c}_2 > 0$ such that for each connection $A=d + a$ on the trivial $SU(r)$-bundle satisfying 
\begin{equation*}
(1) \; d^*a = 0  \qquad
(2) \; \star a|_{\partial \mathbb{B}} = 0 \qquad
(3) \; \|a\|_{L^{7/2}_1(\mathbb{B})} \leq \epsilon_2
\end{equation*}
one has 
\[
\| a \|_{L^q_1(B')} \leq \mathfrak{c}_2 \left(\|F_A\|_{L^2(\mathbb{B})} + \|\pi_7(F_A)\|_{L^q(\mathbb{B})} \right). 
\]
\end{lem}

\begin{proof}
Let's consider the operator $\delta:=d^* \oplus \pi_7 \circ d: \Omega^1(\mathbb{B}, \mathfrak{su}(r)) \to \Omega^0(\mathbb{B}, \mathfrak{su}(r)) \oplus \Omega^2(\mathbb{B}, \mathfrak{su}(r))$. The symbol of this operator is given by $a \mapsto (\iota_{\xi} a, \pi_7(\xi \wedge a))$ when evaluating over a non-zero covector $\xi \in T^*\R^7$. Using the equivalence between $\pi_7(a \wedge \xi) = 0$ and $\xi \wedge a \wedge \star \phi_0 = 0$, one readily checks that the symbol evaluating on $\xi$ is injective when $\xi$ is non-zero. The elliptic regularity estimate tells us 
\[
\|a\|_{L^p_1} \leq const. \left(\|\delta a\|_{L^p} + \|a\|_{L^1} \right)
\]
for any $a \in \Omega^1(\mathbb{B}, \mathfrak{su}(r))$ compactly supported in $\mathbb{B}$. Since the $3$-form $\phi_0$ is exact on $\mathbb{B}$, Stokes' theorem tells us 
\[
2\|\pi_7(da)\|_{L^2}^2 - \|\pi_{14}(da)\|_{L^2}^2 = - 2n \int_{\mathbb{B}} \tr(da \wedge da) \wedge \phi_0 = 0,
\]
from which we conclude that $\|da\|^2_{L^2} = 3 \|\pi_7(da)\|^2_{L^2}$. In particular, we see that every compactly supported $1$-form in $\ker \delta$ has to vanish identically. So the $L^1$-term in the RHS of the elliptic estimate above can be dropped. 

Now let $A = d+a$ be a connection satisfying the assumption, and $\beta: \mathbb{B} \to \R$ a cut-off function with $\beta|_{B'} \equiv 1$. Then 
\begin{equation}\label{e5.1}
|\delta(\beta a)| \leq |d\beta| |a| + |\beta a \wedge \beta a| + |\pi_7(F_A)|. 
\end{equation}
Let's first assume $q \in (0, 14/5]$. Then we have the Sobolev embedding $L^2_1 \hookrightarrow L^q$. Combining with Uhlenbeck's result \autoref{t5.2}, we get 
\[
\|a\|_{L^q} \leq const. \|a\|_{L^2_1} \leq const. \|F_A\|_{L^2}
\]
when $\|F_A\|_{L^{7/2}}$ is sufficiently small. With the elliptic estimate of $\delta$, we get 
\begin{equation}\label{e5.2}
\|\beta a\|_{L^q_1} \leq const. \left( \|F_A\|_{L^2} + \|\beta a \wedge \beta a\|_{L^q} + \|\pi_7(F_A)\|_{L^q} \right). 
\end{equation}
The Hölder's inequality tells us $\|\beta a \wedge \beta a\|_{L^q} \leq \|\beta a\|_{L^{mq}} \|\beta a\|_{L^{nq}}$ for each positive pair $(m, n)$ satisfying $1/m + 1/n =1$. Let $m=7/q$ and $n = 7/(7-q)$. We have the Sobolev embeddings $L^{7/2}_1 \hookrightarrow L^{mq}$ and $L^q_1 \hookrightarrow L^{nq}$, from which it follows that 
\[
\|\beta a \wedge \beta a\|_{L^q} \leq const. \|\beta a\|_{L^{7/2}_1} \|\beta a\|_{L^q_1}. 
\]
When $\|\beta a\|_{L^{7/2}_1}$ is sufficiently small, we can rearrange the quadratic term to the LHS of \eqref{e5.2} to conclude that 
\[
\|a\|_{L^q_1(B')} \leq const. \left(\|F_A\|_{L^2(\mathbb{B})} + \|\pi_7(F_A)\|_{L^q(\mathbb{B})} \right). 
\]

Now for $q \in (14/5, 7)$, we have the Sobolev embedding $L^{14/5}_1 \hookrightarrow L^q$. So \eqref{e5.2} continues to hold with the help of the estimate for $\| \beta a\|_{L^{14/5}_1}$. The same argument then applies to complete the proof. 
\end{proof}

Note that the energy $\mathcal{YM}(A)$ is not conformal invariant. As a compensation, we can consider the normalized energy over balls
\begin{equation}
\hat{\mathcal{E}}(A, B_r(x)):= r^{-3} \int_{B_r(x)} |F_A|^2,
\end{equation}
where $B_r(x)$ is the ball of radius $r$ centered at $x \in X$. It is clear the normalized energy is conformally invariant. The normalized energy enjoys the following monotonicity property, which gives a control on the normalized energy when restricted on smaller balls. 

\begin{lem}\label{l5.4}
Let $(\mathbb{B}, \phi_0)$ be the unit ball in $\R^7$ equipped with the standard $G_2$-structure. There exists a constant $\mathfrak{c}_3 > 0$ such that each connection $A$ on the trivial $SU(n)$-bundle satisfies
\[
\hat{\mathcal{E}}(A, B_r(0)) \leq \hat{\mathcal{E}}(A, \mathbb{B}) + \mathfrak{c}_3 \| \pi_7(F_A) \|^2_{L^{7/2}(\mathbb{B})}. 
\]
\end{lem}

\begin{proof}
Let's write $B_r$ for the ball centered at the origin of radius $r$. The derivative of the normalized energy along the radial direction can be computes as 
\[
\frac{\partial}{\partial r} \hat{\mathcal{E}}(A, B_r(0)) = -3r^{-4} \int_{B_r} |F_A|^2 + r^{-3} \int_{\partial B_r} |F_A|^2. 
\]
Appealing to  the relation $-\tr(F_A \wedge F_A) \wedge \phi_0 = 2|\pi_7(F_A)|^2 - |\pi_{14}(F_A)|^2$, we find that 
\[
\frac{\partial}{\partial r} \hat{\mathcal{E}}(A, B_r(0))  = -9r^{-4}\|\pi_7(F_A)\|^2_{L^2(B_r)} +3r^{-4} \int_{B_r} -2n \tr(F_A \wedge F_A) \wedge \phi_0 + r^{-3} \int_{\partial B_r} |F_A|^2. 
\]
Let $\nu = 1/3 r\partial_r$ be the radial vector field on $\mathbb{B}$. Since $G_2$ acts on the spheres $\partial B_r$ transitively and preserves $\phi_0$, one can verify the relation $d\iota_{\nu} \phi_0 = \phi_0$ directly over points along one axis. Writing $\iota_{\nu} \phi_0 = \omega_0$, we have 
\[
\int_{B_r} \tr(F_A \wedge F_A) \wedge \phi_0 = \int_{\partial B_r} \tr(F_A \wedge F_A) \wedge \omega_0. 
\]
Given any $2$-form $\alpha$, there is a relation 
\begin{equation}\label{e5.4}
3 \alpha \wedge \alpha \wedge \omega_0 \leq |\alpha|^2 r\vol_{\partial B_r}
\end{equation}
on the $r$-sphere $\partial B_r$. To verify this relation, we use the transitivity of the $G_2$-action on the $r$-sphere again, which means it suffices to work at the point $\pmb{x}_0 = (r, 0, ..., 0)$ on the $x_1$-axis where $\vol_{\partial B_r} = dx_{234567}$ and $\omega_0|_{\pmb{x}_0} = 1/3 \cdot x_1(dx_{23} - dx_{45} -dx_{67})$. Then one can write out $\alpha$ in local coordinates to complete the computation. From \eqref{e5.4}, it follows that 
\[
3r^{-4} \int_{\partial B_r} -2n \tr(F_A \wedge F_A) \wedge \omega_0 + r^{-3} \int_{\partial B_r} |F_A|^2 \geq 0. 
\]
In this way, we get a lower bound on the derivative 
\[
\frac{\partial}{\partial r} \hat{\mathcal{E}}(A, B_r(0))  \geq -9r^{-4}\|\pi_7(F_A)\|^2_{L^2(B_r)} \geq const. r  \cdot \|\pi_7(F_A)\|^2_{L^{7/2}(\mathbb{B})}. 
\]
The result now follows from integration from $r$ to $1$. 
\end{proof}

\begin{cor}\label{c5.5}
Let $(\mathbb{B}, \phi_0)$ be the unit ball in $\R^7$ equipped with the standard $G_2$-structure. There exists constants $\epsilon_4 > 0, \mathfrak{c}_4 > 0$ such that for each connection $A$ on the trivial $SU(n)$-bundle satisfying 
\[
\|F_A\|_{L^{7/2}(B_r(x))} \leq \epsilon_4,
\]
where $B_r(x) \Subset B_{1/2}(0)$ is some ball contained in the ball centered at the origin of radius $1/2$, one has 
\[
\|F_A\|_{L^{7/2}(B_{3r/4}(x))} \leq \mathfrak{c}_4 \left(\|F_A\|_{L^2(\mathbb{B})} + \|\pi_7(F_A)\|_{L^{7/2}(\mathbb{B})} \right). 
\]
\end{cor}

\begin{proof}
Let's rescale and translate the ball $B_r(x)$ to the unit ball $\mathbb{B}$, and denote the corresponding connection by $A'=d+a'$. Since the $L^{7/2}$-norm of the curvature and the normalizer energy is conformally invariant, we have 
\[
\|F_A\|_{L^{7/2}(B_r(x))} = \|F_{A'}\|_{L^{7/2}(\mathbb{B})} \quad \text{ and } \quad \hat{\mathcal{E}}(A, B_r(x)) = \hat{\mathcal{E}}(A', \mathbb{B}) = \|F_{A'}\|^2_{L^2(\mathbb{B})}. 
\]
Then \autoref{t5.2} and \autoref{l5.3} tell us that 
\[
\|a'\|_{L^{7/2}_1(B_{3/4}(0))} \leq const. \left(\|F_{A'}\|_{L^2(\mathbb{B})} + \|\pi_7(F_{A'})\|_{L^{7/2}(\mathbb{B})} \right) < 1 
\]
once $\|F_{A'}\|_{L^{7/2}(\mathbb{B})}$ is sufficiently small. Using the Sobolev multiplication $L^{7/2}_1 \times L^{7/2}_1 \hookrightarrow L^{7/2}$, we have 
\[
\|F_{A'}\|_{L^{7/2}(B_{3/4}(0))} \leq const. \left(\|a'\|_{L^{7/2}_1(B_{3/4}(0))} + \|a'\|^2_{L^{7/2}_1(B_{3/4}(0))} \right) \leq const. \|a'\|_{L^{7/2}_1(B_{3/4}(0))}. 
\]
Rescale back to $B_r(x)$, we get 
\[
\|F_A\|_{L^{7/2}(B_{3r/4}(x))} \leq const. \left(r^{-3/2} \||F_A\|_{L^2(B_r(x))} + \|\pi_7(F_A)\|_{L^{7/2}(B_r(x))} \right). 
\]
Note that $B_{r+1/2}(x) \Subset \mathbb{B}$. \autoref{l5.4} gives us 
\[
\begin{split}
r^{-3}\||F_A\|^2_{L^2(B_r(x))} & \leq (r+1/2)^{-3} \| F_A\|^2_{L^2(B_{r+1/2}(x))} + const. \|\pi_7(F_A)\|^2_{L^{7/2}(B_{r+1/2}(x))} \\
& \leq const. \left( \| F_A\|_{L^2(\mathbb{B})} + \|\pi_7(F_A)\|_{L^{7/2}(\mathbb{B})} \right)^2,
\end{split} 
\]
which can be substituted back to the inequality above to complete the proof.  
\end{proof}

Now we are prepared to bound the $L^{7/2}$-norm of the curvature on unit balls. 

\begin{lem}\label{l5.6}
Let $(\mathbb{B}, \phi_0)$ be the unit ball in $\R^7$ equipped with the standard $G_2$-structure and $B' \Subset B_{1/2}(0)$ an interior ball. There are constants $\epsilon_5>0$, $\mathfrak{c}_5>0$ such that for each connection $A$ on the trivial $SU(n)$-bundle satisfying 
\[
\|F_A\|_{L^2(\mathbb{B})} \leq \epsilon_5 \quad \text{ and } \quad \|\pi_7(F_A)\|_{L^{7/2}(\mathbb{B})} \leq \epsilon_5, 
\]
one has the following bound 
\[
\|F_A\|_{L^{7/2}(B')} \leq \mathfrak{c}_5\left(\|F_A\|_{L^2(\mathbb{B})} + \|\pi_7(F_A)\|_{L^{7/2}(\mathbb{B})} \right). 
\]
\end{lem}

\begin{proof}
Given a point $x \in \mathring{B}_{1/2}(0)$, we denote by $D(x) = 1/2 - |x|$ the distance between $x$ and the boundary $\partial B_{1/2}(0)$. Consider 
\[
r(x) = \sup \{ r < D(x): \|F_A\|_{L^{7/2}(B_r(x))} \leq \epsilon_4\},
\]
where $\epsilon_4$ is the constant in \autoref{c5.5}. We claim that when $\|F_A\|_{L^2(\mathbb{B})}$ and $\|\pi_7(F_A)\|_{L^{7/2}(\mathbb{B})}$ are sufficiently small, one can find $M > 0$ such that $D(x) \leq Mr(x)$ holds for all $x \in B_{1/2}(0)$. 

Let's assume the claim is true for a moment. Then given any interior ball $B' \Subset B_{1/2}(0)$, one can cover $B'$ with finitely many balls of the form $B_{3r/4}(x)$ so that $B_r(x) \Subset B_{1/2}(0)$ and $\|F_A\|_{L^{7/2}(B_r(x))} \leq \epsilon_4$. Then we get the desired estimate on $\|F_A\|_{L^{7/2}(B')}$ by applying \autoref{c5.5} to each of the balls. 

Now we prove the claim. Suppose one can find a sequence $x_i \in \mathring{B}_{1/2}(0)$ such that $D(x_i)/r(x_i) \to \infty$. After taking a subsequence, we may assume that $x_i \to x_o$. Note that $x_o \in \partial B_{1/2}(0)$, for otherwise one gets $r(x_o) = 0$ which is impossible. Now we fix $r_o > 0$ so that $\|F_A\|_{L^{7/2}(B_{r_o}(x_o))} < \epsilon_4$, and choose $x_k$ so that $B_{r(x_k)}(x_k) \subset \mathring{B}_{3r_o/4}(x_o)$. \autoref{c5.5} tells us that 
\begin{equation}\label{e5.5}
\|F_A\|_{L^{7/2}(B_{3r_o/4}(x_o))} \leq \mathfrak{c}_4 \left(\|F_A\|_{L^2(\mathbb{B})} + \|\pi_7(F_A)\|_{L^{7/2}(\mathbb{B})} \right). 
\end{equation}
By choosing $\|F_A\|_{L^2(\mathbb{B})}$ and $\|\pi_7(F_A)\|_{L^{7/2}(\mathbb{B})}$ sufficiently small, we can arrange that the RHS of \eqref{e5.5} is less than $\epsilon_4$. In particular, for some $\epsilon > 0$, we get 
\[
\|F_A\|_{L^{7/2}(B_{r(x_k)+\epsilon}(x(k)))} \leq \|F_A\|_{L^{7/2}(B_{3r_o/4}(x_o))} < \epsilon_4. 
\]
This contradicts to the choice of $r(x_k)$. 
\end{proof}

Now we move on to control the higher derivatives of connection $1$-forms.

\begin{lem}\label{l5.7}
Let $(\mathbb{B}, \phi_0)$ be the unit ball in $\R^7$ equipped with the standard $G_2$-structure and $B' \Subset \mathbb{B}$ an interior ball. There is a constant $\epsilon_6 > 0$ such that for each connection $A=d + a$ on the trivial $SU(n)$-bundle satisfying 
\[
d^*a = 0 \quad \text{ and } \quad \|a\|_{L^{7/2}_1(\mathbb{B})} \leq \epsilon_6
\]
one has 
\[
\|a\|_{L^{7/2}_l(B')} \leq  p_l\left(\|F_A\|_{L^2(\mathbb{B})} + \|\pi_7(F_A)\|_{L^{7/2}_{l-1}(\mathbb{B})} \right) \quad \forall \; l \geq 2,
\]
where $p_l: \R \to \R$ is a polynomial with positive coefficients and no zero-th order term. 
\end{lem}

\begin{proof}
Let $\delta:= d^* \oplus \pi_7$ be the operator considered in the proof of \autoref{l5.3}, and $\beta: \mathbb{B} \to \R$ a cut-off function with $\beta|_{B'} \equiv 1$. The elliptic estimates of $\delta$ gives us 
\[
\|\beta a\|_{L^{7/2}_l} \leq const. \left(\|a\|_{L^{7/2}_{l-1}} + \|\beta a \wedge \beta a\|_{L^{7/2}_{l-1}} + \|\pi_7(F_A)\|_{L^{7/2}_{l-1}} \right).
\]
In general, the Leibinz rule gives us the Sobolev inequality:
\[
\|\beta a \wedge \beta a\|_{L^{7/2}_{l-1}} \leq const. \|\beta a\|_{L^{7/2}_1}  \cdot \sum_{k=1}^l \|\beta a\|_{L^{7/2}_k}. 
\]
So when $\|a\|_{L^{7/2}_1}$ is sufficiently small, one can rearrange the term containing $\|\beta a\|_{L^{7/2}_l}$ to get the desired estimate. To see why the constant $\epsilon_6$ can be chosen to be independent of $l$, we note that when $l \geq 4$ the Sobolev multiplication inequality becomes:
\[
\|\beta a \wedge \beta a\|_{L^{7/2}_{l-1}} \leq \|\beta a \|^2_{L^{7/2}_{l-1}}.
\]
So there is no need to do the rearrangement when $l \geq 4$. 
\end{proof}

\begin{proof}[Proof of \autoref{t5.1}]
Let $A$ be a $\pmb{\pi}$-perturbed $G_2$-instanton satisfying the equation $F_A \wedge \psi + \sigma_{\pmb{\pi}}(A) = 0$. We fix a positive number $r_0 > 0$ so such for each $x \in X$ one can find $r_x \geq r_0$ such that the geodesic ball $B_{r_x}(x)$ is contained in an orbifold chart $(\tilde{U}_x, G_x, \varphi_x)$. Let $\tilde{B}_{r_x}(x) \subset \tilde{U}_x$ be a ball containing a connected component of $\varphi_x^{-1}(B_{r_x}(x))$ and $\tilde{A}_x$ the corresponding equivariant connection on $\tilde{U}_x$. 

Since $\|F_A\|^2_{L^2} = 3\|\pi_7(F_A)\|^2_{L^2}$, we know when $\|\pmb{\pi}\|$ is sufficiently small, one can guarantee that 
\[
\hat{\mathcal{E}}(\tilde{A}_x, \tilde{B}_{r_x}(x)) <\epsilon_5 \text{ and } \|\pi_7(F_A)\|_{L^{7/2}(\tilde{B}_{r_x}(x))} < \epsilon_5, 
\]
where $\epsilon_5$ is the constant in \autoref{l5.6}. Then \autoref{l5.6} gives us a bound on $\|F_{\tilde{A}_x}\|_{L^{7/2}(\tilde{B}'_x)}$ for some interior ball $\tilde{B}'_x \Subset \tilde{B}_{r_x}(x)$. We can choose $\|\pmb{\pi}\|$ to be small so that the assumption of \autoref{t5.2} applies. We conclude that after possible gauge transformations over $\tilde{U}_x$, the connection $\tilde{A}_x = d + \tilde{a}_x$ satisfies $d^* \tilde{a}_x = 0$ and 
\[
\|\tilde{a}_x\|_{L^{7/2}_1(\tilde{B}'_x)} \leq const. \left(r_0^{-3} \|F_A\|_{L^2(X)} + \|\pi_7(F_A)\|_{L^{7/2}(X)} \right). 
\]
Now \autoref{l5.7} tells us that 
\begin{equation}\label{e5.6}
\|\tilde{a}_x\|_{L^{7/2}_l(\tilde{B}'_x)} \leq p_l\left(r_0^{-3} \|F_A\|_{L^2(X)} + \|\pi_7(F_A)\|_{L^{7/2}_{l-1}(X)} \right). 
\end{equation}
In order to control $\|\pi_7(F_A)\|_{L^{7/2}_{l-1}(X)}$, we recall that $\pi_7(F_A) = -\star(\psi \wedge \star \sigma_{\pmb{\pi}}(A))$. Then \autoref{l4.2} tells us 
\[
\|\pi_7(F_A)\|_{L^{7/2}_{l-1}(X)} \leq const. \|\pi_7(F_A)\|_{L^2_{l-1}(X)} \leq const. \left(1+\|A^{\dagger} - A^{\dagger}_0\|^{l-1}_{L^2_{l-1}(X)} \right) \|\pmb{\pi}\|,
\]
where the constant depends on $\psi$ and the orbifold chart $(\tilde{U}_x, G_x, \varphi_x)$. A priori, we only have $A^{\dagger} - A^{\dagger}_0 \in L^{7/2}_1$. The recursive inequality \eqref{e5.6} tells us that after a possible gauge transformation $(A^{\dagger} - A^{\dagger}_0)|_{U_x} \in L^{7/2}_2(U_x)$. The gauge patching argument in \cite[Chapter 7]{W05} enables us to find a global gauge transformation $g$ so that $g \cdot A^{\dagger} - A^{\dagger}_0 \in L^{7/2}_2(X)$. Iterating this process, we see that $A^{\dagger} - A^{\dagger}_0 \in L^{7/2}_l$ for any $l$ after possible gauge transformations. 

Now we have proved that the restriction of each $\pmb{\pi}$-perturbed $G_2$-instanton to a small orbifold chart admits a connection $1$-form that is bounded in $L^{7/2}_l$ for any $l$, under the assumption that $\|\pmb{\pi}\|$ is sufficiently small. So given any sequence of $\pmb{\pi}$-perturbed $G_2$-instantons, we can find a subsequence, using a diagonalization procedure, so that the restriction of this subsequence to a given small orbifold chart converges in $C^{\infty}$-topology to a $\pmb{\pi}$-perturbed $G_2$-instanton, after possible gauge transformations. We then cover $X$ with finitely many small orbifold charts, and run the argument over each of them. Finally, we pathch the gauge transformations applied to each charts as in \cite[Chapter 7]{W05} to get the desired convergence on $X$. This proves the compactness of $\M_{\phi,\pmb{\pi}}(X, P^{\dagger})$. 
\end{proof}

\section{\large \bf Deformation Invariants}\label{s6}

Now we come to the point of defining an enumerative invariant for $G_2$-orbifolds by incorporating the techniques developed above. We first pin-down the precise definitions, then calculate the invariants for some orbifolds appeared in Joyce's construction \cite{J96}. 

\subsection{\em Definition of Invariants} \label{ss6.1} \hfill

\vspace{3mm}

\begin{dfn}\label{d6.1}
Let $(X, \phi)$ be a compact $G_2$-orbifold, and $P^{\dagger} \to X$ an admissible principal $U(r)$-orbibundle. 
\begin{enumerate}[label=(\alph*)]
\item When $(X, \phi)$ is simple, and either $r = 2$ or $c_1(P^{\dagger})$ divisible by $2$, we define 
\[
n_{\phi}(X, P^{\dagger}):= \# \M^*_{\phi, \pmb{\pi}}(X, P^{\dagger}) \in \Z
\]
as the signed counting of irreducible $\pmb{\pi}$-perturbed projective $G_2$-instantons with respect to a small generic perturbation $\pmb{\pi}$ and orientation data. Moreover, for fixed $c \in H^2_{\orb}(X; \Z)$, we write 
\[
n^c_{\phi}(X):= \sum_{P^{\dagger}} n_{\phi}(X, P^{\dagger}) \in \Z,
\]
where the sum ranges over all rank-$2$ admissible bundles $P^{\dagger}$ satisfying $c_1(P^{\dagger}) = c$.

\item When $(X, \phi)$ is not simple, we define 
\[
n_{\phi}(X, P^{\dagger}):= \# \M^*_{\phi, \pmb{\pi}}(X, P^{\dagger}) \in \Z/2
\]
as the counting of irreducible $\pmb{\pi}$-perturbed projective $G_2$-instantons with respect to a small generic perturbation $\pmb{\pi}$. 

\item When $(X, \phi)$ is simple and satisfies $H^1_{\orb}(X; \Z) = 0$, and $P^{\dagger}$ is an $SU(2)$-bundle with $c_2(P^{\dagger}) = 0 \in H^4_{\orb}(X;\Z)$, we define 
\[
n_{\phi}(X, P^{\dagger}):= \# \mathcal{M}^*_{\phi, \pmb{\pi}}(X, P^{\dagger}) \in \Z
\] 
as the signed counting of irreducible $\pmb{\pi}$-perturbed (non-projective) $G_2$-instantons with respect to a small generic perturbation $\pmb{\pi}$ and orientation data. Moreover, we write 
\[
n_{\phi}(X):= \sum_{P^{\dagger}} n_{\phi}(X, P^{\dagger}) \in \Z,
\]
where the sum ranges over all $SU(2)$-bundles over $X$ with vanishing $c_2$. 
\end{enumerate}
\end{dfn}

The items in \autoref{d6.1} require some explanation. In (a), we define $n^c_{\phi}(X)$ as the sum of admissible bundles with $c_1 =c$. The requirement $c^2_1 - 4c_2 = 0$ implies that there are only finitely many choices of $c_2$. Note that the higher homotopy groups $\pi_n(U(2))$ with $4 \leq n \leq 7$ are all finite. Thus obstruction theory tells us there are only a finite number of $U(2)$-bundles over $X$ once $c_1$ and $c_2$ are determined. So the sum defining $n^c_{\phi}(X)$ is finite. In $(c)$, the requirement $H^1_{\orb}(X; \Z) = 0$ implies that there are no non-trivial flat reducible $SU(2)$-connections, and the unique reducible flat $SU(2)$-connection is isolated from the irreducible ones. These properties continues to hold if the perturbation $\pmb{\pi}$ is chosen to be sufficiently small. Thus the counting that defines $n_{\phi}(X)$ makes sense. 

As for the orientation data mentioned in \autoref{d6.1}, we refers to Joyce--Upmeier \cite{JU21} for the manifold case where they considered `flag structures' as additional inputs. We believe such a work carries through the orbifold case after unraveling the corresponding homotopy theory. We choose not to illustrate the orientation data here for the reason of lacking practical applications. If the moduli space is non-empty, one can alway fix an orientation at one point, and use it to orient other points via the spectral flow of the deformation operator along a path. 

\begin{thm}\label{t6.2}
All invariants in \autoref{d6.1} are independent of the choices of the perturbation $\pmb{\pi}$, and invariant under $C^0$-deformation of torsion-free $G_2$-structures. 
\end{thm} 

\begin{proof}
The independence of perturbations follows from a standard argument. Given two small generic perturbations $\pmb{\pi}_0$ and $\pmb{\pi}_1$ so that the moduli spaces $\M_{\phi, \pmb{\pi}_i}(X, P^{\dagger})$ are both regular. We can choose a path of small perturbations $\gamma: [0, 1] \to \mathscr{P}$, connecting $\pmb{\pi}_0$ and $\pmb{\pi}_1$, transverse to the image $\pr_1(\mathcal{F}^{-1}(0))$, where $\mathcal{F}$ is the map defined in \eqref{e4.10}, and $\pr_1: \mathscr{P} \times \A_k \to \mathscr{P}$ is the projection onto the first factor. Then transversality tells us that 
\[
\M^*_{\phi, \gamma}:= \bigcup_{t \in [0, 1]} \M^*_{\phi, \gamma(t)}(X, P^{\dagger})
\]
gives an (oriented) cobordism from $\M^*_{\phi, \pmb{\pi}_0}(X, P^{\dagger})$ to $\M^*_{\phi, \pmb{\pi}_1}(X, P^{\dagger})$. Thus the counting with respect to $\pmb{\pi}_0$-instantons and $\pmb{\pi}_1$-instantons coincides. 

Suppose now we have a continuous path $\Phi:=(\phi_t)_{t\in [0,1]}$ of torsion-free $G_2$-structures on $X$. Then we consider 
\[
\M_{\Phi}(X, P^{\dagger}) : = \bigcup_{t \in [0, 1]} \M_{\phi_t}(X, P^{\dagger}) \subset [0, 1] \times \A_k/\G_{k+1}. 
\]
We can consider the map analogous to \eqref{e4.10}: 
\[
\begin{split}
\mathcal{F}_{\Phi}: \mathscr{P} \times \A^*_k \times [0, 1] & \longrightarrow \mathcal{K}_{k-1} \\
(\pmb{\pi}, A, t) & \longmapsto \grad \cs_{\phi_t, \pmb{\pi}}(A) 
\end{split}
\]
Due to our assumption, $\mathcal{F}_{\Phi}$ is smooth on the first two components and only continuous on the third. However, \autoref{l4.4} tells us that the partial derivatives of $\mathcal{F}_{\Phi}$ on the first two components are already transverse to the zero section of $\mathcal{K}_{k-1}$. Thus the implicit function theorem tells us for a generic $\pmb{\pi} \in \mathscr{P}$, the perturbed moduli space $\M_{\Phi, \pmb{\pi}}(X, P^{\dagger})$ is a $C^0$ $1$-manifold. Since each $\phi_t$ is torsion-free, \autoref{t5.1} tells us that $\M_{\Phi, \pmb{\pi}}(X, P^{\dagger})$ is compact when $\pmb{\pi}$ is sufficiently small. Thus we get an (oriented) cobordism from $\M_{\phi_0}(X, P^{\dagger})$ to $\M_{\phi_1}(X, P^{\dagger})$, which implies the deformation invariance. Here the oriented corbordism means on a neighborhood of $\M_{\Phi}(X, P^{\dagger})$ in $[0, 1] \times \A_k/\G_{k+1}$, the determinant line of the index bundle of the deformation complex is trivialized so that the orientation is transported inside this neighborhood rather than `along the $C^0$ curve'. 
\end{proof}

Recall in the introduction section, an irreducible projectively flat connection $A^{\dagger}$ of $P^{\dagger}$ is said to be non-degenerate if $H^1(X, \rho_A) = 0$, where $\rho_A$ is the representation $\rho_A: \pi^{\orb}_1(X) \to SO(r^2 - 1)$ given by the induced flat connection on the bundle $\su(P^{\dagger})$. Equivalently, if we write $\ad_A$ for the local system arising as $\g_P$ twisted by $A$, then $H^1(X, \ad_A) = H^1(X, \rho_A)$. \autoref{p1.5} follows from the following observation.

\begin{lem}\label{l6.3}
Let $[A^{\dagger}]$ be a gauge class of irreducible projectively flat connection on an admissible bundle $P^{\dagger}$ of a $G_2$-orbifold $(X, \phi)$. Then $[A^{\dagger}]$ is regular as a point in the moduli space $\M^*_{\phi}(X, P^{\dagger})$ if and only if $[A^{\dagger}]$ is a non-degenerate projectively flat connection.
\end{lem}

\begin{proof}
Recall for the deformation complex $E_A(X)$, irreducibility of $A$ implies that 
\[
H^0(E_A(X)) = H^3(E_A(X)) = 0 \qquad H^1(E_A(X)) \simeq H^2(E_A(X)).
\] 
To prove the lemma, it suffices to identify $H^1(E_A(X))$ with $H^1(X, \ad_A)$. We claim that $\forall a \in \Omega^1(X, \g_P)$ 
\[
d_A a \wedge \psi = 0 \Longleftrightarrow d_A a = 0.
\]
To prove the claim, suppose $d_A a \wedge \psi = 0$, equivalently $\pi_7(d_A a) = 0$ by \autoref{l3.3}. Then \eqref{e2.7} tells us that 
\[
|\pi_{14}(d_A a)|^2 \vol  = 2r \tr(d_A a \wedge d_A a) \wedge \phi) = 2r d \left(\tr(a \wedge d_A a) \wedge \phi)\right),
\]
where the last equality uses the fact that $A$ is flat. Integrating over $X$ gives us that $\pi_{14}(d_A a) = 0$. Thus $d_A a = 0$. This proves the claim. The identification between $H^1(E_A(X))$ and $H^1(X, \ad_A)$ follows from the claim directly. 
\end{proof}

\begin{proof}[Proof of \autoref{p1.5}]
\autoref{l6.3} tells us that $\M_{\phi}(X, P^{\dagger})$ is regular for any torsion-free $G_2$-structure $\phi$ under the given assumptions. Thus we can choose $\pmb{\pi} = 0$ for defining $n_{\phi}(X, P^{\dagger})$. In this case, the cohomology of the deformation complexes are identified for all torsion-free $G_2$-structures canonically by \autoref{l6.3}. Thus the counting with signs agrees for different torsion-free $G_2$-structures. 
\end{proof}

When the $G_2$-orbifold is flat, Walpuski \cite{W13} gave a simpler criterion to justify non-degeneracy by appealing to the Weitzenböck formula of Dirac operators. Suppose $A^{\dagger}$ is a projectively flat connection. Then the flat Levi--Civita connection on $X$ coupled with the flat connection $\ad_A$ on $\g_E$ gives a representation $\rho^{\LC}_A: \pi_1^{\orb}(X) \to \Aut((\Lambda^1 \otimes \g_P)|_{x_o})$ for a reference non-singular point $x_o \in X$.  

\begin{lem}[{\cite[Proposition 9.2]{W13}}]\label{l6.4}
Let $(X, \phi)$ be a flat $G_2$-orbifold, and $A^{\dagger}$ a projectively flat connection on a $U(r)$-bundle $P^{\dagger}$ of $X$. Then $A^{\dagger}$ is non-degenerate if and only if the representation $\rho^{\LC}_A$ fixes no non-trivial vectors in $(\Lambda^1 \otimes \g_P)|_{x_o}$. 
\end{lem}

\begin{cor}\label{c6.5}
Let $P^{\dagger}$ be an admissible bundle over a flat torsion-free $G_2$-orbifold $(X, \phi)$. Then the moduli space of irreducible projective instantons $\M^*_{\phi}(X, P^{\dagger})$ is regular. 
\end{cor}

\begin{proof}
Since every irreducible projective $G_2$-instanton $A^{\dagger}$ induces a flat connection $A$ on $\g_P$, we know that, after coupling with the flat Levi--Civita connection, $\rho^{LC}_A$ fixes no non-trivial vectors. 
\end{proof}

\subsection{\em Calculation of Invariants} \label{ss6.2} \hfill

\vspace{3mm}

In this subsection, we calculate the invariant $n_{\phi}(X, P^{\dagger})$ for some examples. 

\vspace{3mm}

\begin{exm}\label{ex6.1}
The first example we consider is the manifold $X = T^3 \times K3$ whose $G_2$-structure is induced from the metric given by the product of a flat metric on $T^3$ with a hyperkähler metric on $K3$. We may identify $T^3=S^1 \times S^1 \times S^1$, and write $\alpha, \beta, \gamma$ for embedded curves in $T^3$ representing the three factors respectively. Let $Q^{\dagger} \to T^3$ be the $U(2)$-bundle over $T^3$ with $c_1(Q^{\dagger}) = \PD[\gamma]$ and $c_2(Q^{\dagger}) = 0$, and $P^{\dagger} = \pr_1^*Q^{\dagger}$ the pull-back of $Q^{\dagger}$ to $X$ via the projection map. The bundle $P^{\dagger}$ is admissible since $\langle c_1(P^{\dagger}), \alpha \times \beta \rangle = 1$. 

Then the unperturbed moduli space corresponds to the representation variety:
\[
\mathcal{R}(P^{\dagger}) = \{\rho: \pi_1(T^3 \backslash \gamma) \to SU(2): \rho([\alpha, \beta]) = -1\}/\Ad.
\] 
After identifying $SU(2)$ with unit quaternions, one readily solves that $\rho(\alpha) = i$, $\rho(\beta) = j$, and $\rho(\gamma) = \pm 1$ up to the adjoint action. Thus the unperturbed moduli space $\M_{\phi}(X, P^{\dagger})$ consists of two points, say $\rho_{\pm}$, corresponding to $\rho_{\pm}(\gamma) = \pm 1$ respectively. To see why both $\rho_{\pm}$ are regular, we note that $P^{\dagger}$ is pulled-back from $Q^{\dagger}$ on $T^3$, so the Künneth formula tells us that 
\[
H^1(X, \ad_{\rho_{\pm}}) \simeq H^1(T^3, \ad_{\rho_{\pm}}) = 0. 
\]
Then invoking \autoref{l6.3}, we get the non-degeneracy of $\rho_{\pm}$. \cite[Lemma 6.3]{K04} tells us that one can find a path of connections on $Q^{\dagger}$ that joins $\rho_{\pm}$ whose spectral flow of the deformation operator for the flatness equation on $Q^{\dagger}$ is $4$. Over $T^3 \times K3$ the deformation complex $E_A(X)$ splits as the sum of the deformation complexes of the flat connections on $T^3$ and $ASD$ connections on $K3$ for pull-back connections. Since $K3$ is complex, if we pull-back the path above, the resulting spectral flow of the deformation operator of $G_2$-instantons joining $\rho_{\pm}$ is still even. After fixing an orientation on either of $\rho_{\pm}$, we conclude that 
\[
n_{\phi}(T^3 \times K3, P^{\dagger}) = 2. 
\]

In general, following the construction of Kronheimer \cite[Section 6.2]{K04}, one can find $U(r)$-bundles $P^{\dagger}_r$ over $T^3 \times K3$ so that $n_{\phi}(T^3 \times K3, P^{\dagger}_r) = r$ using the argument above. 
\end{exm}

\vspace{3mm}

\begin{exm}
The second example we consider arises from Example 3 in Joyce's list \cite{J96}. Let $T^7$ be the torus obtained by quotienting the standard $\R^7$ with axis-translations. Let's write $(x_1, x_2, ..., x_7)$ for the coordinates on $\R^7$. We denote the $x_i$-translation by $\tau_i: x_i \mapsto x_i+1$, and consider three involutions on $\R^7$ that all preserve the flat structure:
\begin{equation}
\begin{split}
\alpha: (x_1, ..., x_7) & \longmapsto (x_1, x_2, x_3, -x_4, -x_5, -x_6, -x_7) \\
\beta: (x_1, ..., x_7) & \longmapsto (x_1, -x_2, -x_3, x_4, x_5, \frac{1}{2}-x_6, -x_7) \\
\gamma: (x_1, ..., x_7) & \longmapsto (-x_1, x_2, -x_3, x_4, \frac{1}{2}-x_5, x_6, \frac{1}{2}-x_7) \\
\end{split}
\end{equation}
We write $\Gamma$ for the group generated by $\tau_1, ..., \tau_7$ and $\alpha, \beta, \gamma$. We let $X = \R^7/\Gamma$ equipped with the flat metric. From \cite[Lemma 12.2.2]{J00}, we know $X$ is a simple orbifold whose singular set consists of 12 disjoint copies of $T^3$. Since $\R^7$ is the universal cover of $X$, we know the orbifold fundamental group $\pi^{\orb}_1(X) = \Gamma$. The relations of generators in $\Gamma$ are generated by the following:
\begin{enumerate}[label=(\arabic*)]
\item $[\tau_i, \tau_j] = 1 \;\;  \forall i,j =1, ..., 7$.
\item $\alpha^2 = 1$, $\beta^2 = 1$, $\gamma^2 = 1$. 
\item $[\alpha, \beta] = \tau_6^{-1}$, $[\alpha, \gamma] = \tau_7^{-1}$, $[\beta, \gamma] = \tau_7^{-1}$. 
\item $\alpha\tau_i\alpha = \tau_i$ for $i = 1, 2, 3$, \; $\alpha \tau_j \alpha = \tau_j^{-1}$, for $j = 4, 5, 6, 7$. 
\item $\beta \tau_i \beta = \tau_i$ for $i=1, 4, 5$, \; $\beta \tau_j \beta = \tau_j^{-1}$, for $j=2, 3, 6, 7$. 
\item $\gamma \tau_i \gamma = \tau_i$, for $i=2, 4, 6$, \; $\gamma \tau_j \gamma = \tau_j^{-1}$ for $j = 1, 3, 5, 7$. 
\end{enumerate}

Now we determine all $SO(3)$-representations $\Hom(\pi^{\orb}_1(X), SO(3))$ up to conjugation based on the relations above. Let's write $V_4 = \{1, a, b, c\} \subset SO(3)$ for the Klein four-group given by
\[
a = 
\begin{pmatrix}
1 & & \\
& -1 & \\
& & -1 
\end{pmatrix}, \quad 
b = 
\begin{pmatrix}
-1 & & \\
& 1 & \\
& & -1 
\end{pmatrix}, \quad 
c = 
\begin{pmatrix}
-1 & & \\
& -1 & \\
& & 1 
\end{pmatrix}.
\]
We first recall an elementary fact in linear algebra and leave it as an exercise for the reader. 

\begin{lem}\label{l6.4}
For a nontrivial element $1 \neq g \in SO(3)$, we write $L_g$ for the rotation axis of $g$. Then 
\begin{enumerate}
\item If $aga = g$, then $g \in O(2)$ of the following form:
\[
g = 
\begin{pmatrix}
1 & & \\
& \cos \theta & \sin \theta \\
& - \sin \theta & \cos \theta
\end{pmatrix},  
 \quad
\begin{pmatrix}
 -1 & & \\
 & \cos \theta & \sin \theta \\
 & \sin \theta & - \cos \theta
\end{pmatrix}. 
\]

\item If $aga  = g^{-1}$, then either $g = a$ or $g = ar$ with $r \in SO(3)$ a $\pi$-rotation. In the latter case, we orient $L_r$ and $L_a$ by $\vec{o}_r$ and $\vec{o}_a$ so that the angle between them satisfies $\theta_{ar} \in [0, \pi/2]$. Then $g$ is the counterclockwise $2\theta_{ar}$-rotation about the axis $L_{ar}$ determined by $L_{ar} \perp L_a$ and $L_{ar} \perp L_r$. Here counterclockwise rotation is defined with respect to the oriented basis $(\vec{o}_r, \vec{o}_a)$. 
\end{enumerate}
Moreover, let's write $S(2) \subset SO(3)$ for the subset of $SO(3)$ consisting of elements satisfying (2). We have the following relations among $O(2)$, $S(2)$ and $V_4$. 
\begin{enumerate}[label=(\alph*)]
\item If $g^2 = 1$ and $g \in O(2) \backslash \{1, a\}$, then $\{1, a, g, ag\}$ forms a copy of $V_4$. 
\item If $g^2 = 1$ and $g \in S(2) \backslash \{1, a\}$, then $\{1, a, g, ag\}$ forms a copy of $V_4$. 
\item If $g \in O(2) \cap S(2) \backslash \{1, a\}$, then $\{1, a, g, ag\}$ forms a copy of $V_4$. 
\end{enumerate}
In particular, if $g \in SO(3)$ satisfies one of (a)--(c), then $g$ is conjugate to $b$ via an element in $SO(2)$. 
\end{lem}

We write $\rho: \pi^{\orb}_1(X) \to SO(3)$ for a general representation.

\vspace{2mm}

{\noindent \bf Case (1)} Suppose $\rho(\alpha) = \rho(\beta) = \rho(\gamma) = 1$. Then we know $\rho(\tau_6) = \rho(\tau_7) = 1$ by relation (3). Then relations (1) and (4)--(6) imply that $\rho(\tau_i) \in V_4$, $i=1, ..., 5$, up to conjugation. 

\vspace{2mm}

{\noindent \bf Case (2)} Suppose $\rho(\alpha) = \rho(\beta) = 1$ and $\rho(\gamma) \neq 1$. Up to conjugation, we can take $\rho(\gamma) = a$. Then we know $\rho(\tau_6) = \rho(\tau_7) = 1$ as before. After further gonjugating by an element in $SO(2)$,  relations (4) and (5) imply that $\rho(\tau_i) \in V_4$ for $i = 2, 3, 4, 5$. 

\vspace{1mm}

{\em Case (2.1)} If $\rho(\tau_i) = 1$ for all $i \in \{2, 3, 4, 5\}$, then $\rho(\tau_1) \in S(2)$. 

{\em Case (2.2)} If not all of $\rho(\tau_i)$, $i \in \{2, 3, 4, 5\}$, are trivial, then $\rho(\tau_1) \in V_4$. 

\vspace{2mm}

{\noindent \bf Case (3)} Suppose $\rho(\alpha) = \rho(\gamma) = 1$, and $\rho(\beta) \neq 1$. Up to conjugation, we can take $\rho(\beta) = a$. Then we know $\rho(\tau_6) = \rho(\tau_7) = 1$. After further gonjugating by an element in $SO(2)$,  relations (4)--(6) imply that $\rho(\tau_i) \in V_4$ for $i=1, 3, 4, 5$. 

\vspace{1mm}

{\em Case (3.1)} If $\rho(\tau_i) = 1$ for all $i \in \{1, 3, 4, 5\}$, then $\rho(\tau_2) \in S(2)$.

{\em Case (3.2)} If not all of $\rho(\tau_i)$, $i \in \{1, 3, 4, 5\}$, are trivial, then $\rho(\tau_2) \in V_4$. 

\vspace{2mm}

{\noindent \bf Case (4)} Suppose $\rho(\beta) = \rho(\gamma) = 1$, and $\rho(\alpha) \neq 1$. Up to conjugation, we can take $\rho(\alpha) = a$. Then we know $\rho(\tau_6) = \rho(\tau_7) = 1$. After further gonjugating by an element in $SO(2)$,  relations (4)--(6) imply that $\rho(\tau_i) \in V_4$ for $i=1, 2, 3, 5$. 

\vspace{1mm}

{\em Case (4.1)} If $\rho(\tau_i) = 1$ for all $i \in \{1, 2, 3, 5\}$, then $\rho(\tau_4) \in S(2)$.

{\em Case (4.2)} If not all of $\rho(\tau_i)$, $i \in \{1, 2, 3, 5\}$, are trivial, then $\rho(\tau_4) \in V_4$. 

\vspace{2mm}

{\noindent \bf Case (5)} Suppose $\rho(\alpha) = 1$, and $\rho(\beta), \rho(\gamma) \neq 1$. Up to conjugation, we can take $\rho(\beta) = a$. Then we know $\rho(\tau_6) = \rho(\tau_7) = 1$. After further gonjugating by an element in $SO(2)$,  relation (3) implies that $\rho(\gamma) \in V_4$. Relations (4) and (5) imply that $\rho(\tau_i) \in V_4$ for $i=4, 5$. 

\vspace{1mm}

{\em Case (5.1)} If $\rho(\gamma) = \rho(\tau_i) = 1$ for all $i \in \{4, 5\}$, then relation (6) implies that $\rho(\tau_j) \in V_4$ for $j = 1, 3$. If $\rho(\tau_1) = \rho(\tau_3) = 1$, then $\rho(\tau_2) \in S(2)$. Otherwise $\rho(\tau_2) \in V_4$. 

{\em Case (5.2)} If $\rho(\gamma) \neq 1$, then relation (6) implies that $\rho(\tau_2), \rho(\tau_1) \in V_4$. If $\rho(\tau_2) = \rho(\tau_1) = 1$ and $\rho(\gamma) = a$, then $\rho(\tau_3) \in S(2)$. Otherwise $\rho(\tau_3) \in V_4$. 

{\em Case (5.3)} If $\rho(\gamma) = 1$, and at least one of $\rho(\tau_4)$, $\rho(\tau_5)$ is non-trivial, then relation (6) implies that $\rho(\tau_1), \rho(\tau_3) \in V_4$. If $\rho(\tau_i) \in \{1, a\}$ for all $i=1, 3, 4, 5$, then $\rho(\tau_2) \in S(2)$. Otherwise $\rho(\tau_2) \in V_4$. 
 
\vspace{2mm}

{\noindent \bf Case (6)} Suppose $\rho(\beta) = 1$, and $\rho(\alpha), \rho(\gamma) \neq 1$. Up to conjugation, we can take $\rho(\alpha) = a$. Then we know $\rho(\tau_6) = \rho(\tau_7) = 1$. After further gonjugating by an element in $SO(2)$,  relation (3) implies that $\rho(\gamma) \in V_4$. Relations (4) and (5) imply that $\rho(\tau_i) \in V_4$ for $i=2, 3$. 

\vspace{1mm}

{\em Case (6.1)} If $\rho(\gamma) = \rho(\tau_i) = 1$ for all $i \in \{2, 3\}$, then relation (6) implies that $\rho(\tau_j) \in V_4$ for $j = 1, 4$. If $\rho(\tau_1) = \rho(\tau_4) = 1$, then $\rho(\tau_5) \in S(2)$. Otherwise $\rho(\tau_5) \in V_4$. 

{\em Case (6.2)} If $\rho(\gamma) \neq 1$, then relation (6) implies that $\rho(\tau_1), \rho(\tau_4) \in V_4$. If $\rho(\tau_1) = \rho(\tau_4) = 1$ and $\rho(\gamma) = a$, then $\rho(\tau_5) \in S(2)$. Otherwise $\rho(\tau_5) \in V_4$. 

{\em Case (6.3)} If $\rho(\gamma) = 1$, and at least one of $\rho(\tau_2)$, $\rho(\tau_3)$ is non-trivial, then relation (6) implies that $\rho(\tau_1), \rho(\tau_4) \in V_4$. If $\rho(\tau_i) \in \{1, a\}$ for all $i=1, 2, 3, 4$, then $\rho(\tau_5) \in S(2)$. Otherwise $\rho(\tau_5) \in V_4$. 

\vspace{2mm}

{\noindent \bf Case (7)} Suppose $\rho(\gamma) = 1$, and $\rho(\alpha), \rho(\beta) \neq 1$. Up to conjugation, we can take $\rho(\alpha) = a$. Then we know $\rho(\tau_7) = 1$. After further gonjugating by an element in $SO(2)$, relations (1) and (6) imply that $\rho(\tau_i) \in V_4$ for $i=1 ,3 ,5$. 

\vspace{1mm}

{\em Case (7.1)} If $\rho(\tau_i) = 1$ for all $i \in \{1, 3, 5\}$, and $\rho(\beta) = a$, then we know $\rho(\tau_6) = 1$. After conjugating by an elements in $SO(2)$ (which we can arrange because $\rho(\tau_i) = 1$ for $i=1, 3, 5$), relations (4) and (5) imply that $\rho(\tau_2), \rho(\tau_4) \in V_4$. 

{\em Case (7.2)} If $\rho(\tau_i) = 1$ for all $i \in \{1, 3, 5\}$, and $\rho(\beta) \neq a$, we know $L_{\alpha \beta} \perp L_{\alpha}, L_{\beta}$, where $L$ means the rotation axis of the corresponding element in $SO(3)$ under $\rho$. Note that 
\begin{equation}\label{e6.2}
\beta\tau_2 \beta = \tau^{-1}_2 \qquad \alpha \tau_2 \alpha = \tau_2.
\end{equation}
We know that $\rho(\tau_2)$ is the $\pi$-rotation about $L_{\alpha \beta}$. The analogous relations of $\tau_4$ implies that $\rho(\tau_4)$ is also the $\pi$-rotation about $L_{\alpha \beta}$. Conjugating by elements in $SO(2)$, we can arrange that $\rho(\tau_2) = \rho(\tau_4) = b$. Then relation (5) implies that $\rho(\beta) \in V_4$, thus $\rho(\tau_6) = 1$. 

{\em Case (7.3)} If $\rho(\tau_i) \in \{1, a\}$ for all $i \in \{1, 3, 5\}$, and not all of them are trivial, relation (5) implies that $\rho(\beta) \in V_4$, thus $\rho(\tau_6) = 1$. Depending on $\rho(\beta) = a$ or not, the argument in the two cases above implies that $\rho(\tau_2), \rho(\tau_4) \in V_4$. 

\vspace{2mm}

{\noindent \bf Case (8)}  Suppose $\rho(\alpha), \rho(\beta), \rho(\gamma) \neq 1$. Up to conjugation, we can take $\rho(\alpha) = a$. 

\vspace{1mm}

{\em Case (8.1)} If $\rho(\beta) = \rho(\gamma) = a$, then relations (4)--(6) imply that, after further conjugating by an element in $SO(2)$, $\rho(\tau_i) \in V_4$ for all $i=1, ..., 5$ and $\rho(\tau_6) = \rho(\tau_7) = 1$.  

{\em Case (8.2)} If $\rho(\beta) = a$, $\rho(\gamma) \neq a$, then relations (4) and (5) imply that $\rho(\tau_i) \in V_4$ for $i=2, 3, 4, 5$, after further conjugating by an element in $SO(2)$. Relation (3) implies that $\rho(\tau_6) = 1$. Note that 
\[
\alpha \tau_1 \alpha = \tau_1 \qquad \gamma \tau_1 \gamma = \tau_1^{-1}.
\]
We know that $\rho(\tau_1)$ is the $\pi$-rotation about $L_{\alpha\gamma}$. If $\rho(\tau_i) = 1$ for $i=2, 3, 4, 5$, we can apply the $SO(2)$-conjugation here to arrange that $\rho(\tau_1) = b$. Otherwise relation (1) implies that $\rho(\tau_1) \in \{b, c\} \subset V_4$. Then relation (6) implies that $\rho(\gamma) \in \{b, c\}$, thus $\rho(\tau_7) = 1$. 

{\em Case (8.3)} If $\rho(\beta) \neq a$, $\rho(\gamma) = a$, then relations (4) and (6) imply that $\rho(\tau_i) \in V_4$ for $i = 1, 3, 4, 6$, after further conjugating by an element in $SO(2)$. Relation (3) implies that $\rho(\tau_7) = 1$. Note that 
\[
\alpha\tau_2\alpha = \tau_2 \qquad \beta \tau_2 \beta = \tau_2^{-1}.
\]
We know that $\rho(\tau_2)$ is the $\pi$-rotation about $L_{\alpha\beta}$. The same argument in Case (8.2) tells us that $\rho(\tau_2) \in \{b, c\}$. Then relation (5) implies that $\rho(\beta) \in \{b, c\}$, thus $\rho(\tau_6) = 1$. 

{\em Case (8.4)} If neither of $\rho(\beta)$ nor $\rho(\gamma)$ equals $a$, we consider 
\begin{equation}
\alpha \tau_1 \alpha = \tau_1 \qquad \beta \tau_1 \beta = \tau_1 \qquad \gamma \tau_1 \gamma = \tau^{-1}_1. 
\end{equation}
We know that $\rho(\tau_1)$ is the $\pi$-rotation about $L_{\alpha\gamma}$. If $\rho(\beta) \neq \rho(\gamma)$, we know $L_{\alpha\gamma} = L_{\beta\gamma} \perp L_{\alpha}, L_{\beta}, L_{\gamma}$, which is a contradiction. Thus $\rho(\beta) = \rho(\gamma)$. Conjugating by an element in $SO(2)$, we can arrange that $\rho(\tau_1) = b$. Then $[\beta, \tau_1] = 1$ tells us that $\rho(\beta) = \rho(\gamma) \in V_4$. Then relation (3) tells us that $\rho(\tau_6) = \rho(\tau_7) = 1$. This argument can be applied to $\tau_i$ for $i=2, 3, 4, 5$ to conclude that $\rho(\tau_i)$ is the $\pi$-rotation about $L_{\alpha\beta}$. Now we know 
\[
[b, \rho(\tau_i)] = [\rho(\tau_1), \rho(\tau_i)] = 1 = [\rho(\alpha), \rho(\tau_i)] = [a, \rho(\tau_i)].
\]
This implies that $\rho(\tau_i) \in V_4$ for $i = 2, 3, 4, 5$. 

\vspace{2mm}

Throwing away the reducible representations above, we obtain the following characterization of irreducible representations.

\begin{cor}
Up to conjugation, each irreducible representation $\rho: \pi_1^{\orb}(X) \to SO(3)$ takes the form
\[
\rho(\alpha), \rho(\beta), \rho(\gamma) \in V_4, \quad \rho(\tau_i) \in V_4, \; i = 1, ..., 5, \quad \rho(\tau_6)=\rho(\tau_7) = 1
\]
with $\im \rho = V_4$. 
\end{cor}

Our next task is to identify admissible bundles that correspond to these irreducible representations. Rather than exhausting all possibilities, we work out a few examples on which there is a unique irreducible representation up to conjugation. Let's write $\varpi: T^7 \to X$ for the quotient map, and $P$ an $SO(3)$-bundle on $X$ induced from an admissible bundle $P^{\dagger}$. Then we get an $SO(3)$-bundle $\varpi^*P$ on $T^7$. \cite[Proposition 11.1]{GM93} implies that every non-trivial flat $SO(3)$-bundle on $T^7$ is classified by its second Stiefel--Whitney class $w_2 \in H^2(T^7; \Z/2)$, and each such bundle admits a unique flat connection up to gauge. In our case, we require $\rho(\tau_6) = \rho(\tau_7) = 1$. Thus if $\varpi^*P$ is  non-trivial, then it's determined by the values $\rho(\tau_i) \in V_4$, $i=1, ..., 5$, up to permutations of the non-trivial elements $\{a, b, c\} \subset V_4$. 

To pin-down the bundle $P$ on $X$, we consider the singular set $S_X$ consisting of 12 copies of $T^3$. Explicitly, each of the involutions $\alpha, \beta, \gamma$ contributes to $4$ copies of $T^3$ as the image of its fixed point set in $T^7$. The fiber of the normal bundle of each $T^3$ is a cone over $\R P^3$, thus contributes to a linking circle corresponding to the non-trivial element in $\pi_1(\R P^3)$. We can write down the elements represented by these linking circles explicitly in $\Gamma = \pi_1^{\orb}(X)$ as follows.

\begin{enumerate}[label=(\alph*)]
\item Linking circles for $\alpha$-tori: $\alpha$, $\tau_4\alpha$, $\tau_5 \alpha$, $\tau_7 \alpha$.

\item Linking circles for $\beta$-tori: $\beta$, $\tau_2\beta$, $\tau_3\beta$, $\tau_2\tau_3\beta$.

\item Linking circles for $\gamma$-tori: $\gamma$, $\tau_1 \gamma$, $\tau_3\gamma$, $\tau_1\tau_3\gamma$. 
\end{enumerate}

Let's write $\Xi \subset \Gamma$ for the subset consisting of the 12 elements corresponding to these linking circles. A necessary condition for two representations $\rho_1, \rho_2: \pi_1^{\orb}(X) \to SO(3)$ to produce the same flat bundle $P$ is that $\ker \rho_i \cap \Xi$ must coincide. Because $\rho(g) \neq 1 \in V_4$ for some $g \in \Xi$ means that the restriction of $w_2(P)$ to a linking $\R P^3$ of the $T^3$ corresponding to $g$ is non-trivial. 

Now we write down a few representations whose corresponding bundles each admits one flat connection up to $SO(3)$-gauge transformations. One can certainly write down a lot more appealing to the discussion above. Similar calculations can be worked out for other examples in Joyce's list as long as one considers the quotient of $T^7$ by involutions. 

\begin{table}[h]
\begin{tabular}{ | M{1cm} | M{1cm} | M{1cm} | M{1cm} | M{1cm} | M{1cm} | M{1cm} | M{1cm} | M{1cm} |} 
 \hline
 & $\alpha$ & $\beta$ & $\gamma$ & $\tau_1$ & $\tau_2$ & $\tau_3$ & $\tau_4$ & $\tau_5$ \\ 
 \hline
$\rho_1$ & 1 & 1 & 1 & $c$ & $b$ & 1 & $a$ & 1 \\
 \hline
 $\rho_2$ & $a$ & 1 & 1 & $c$ & $b$ & 1 & $a$ & $a$  \\
 \hline
  $\rho_3$ & $a$ & $b$ & 1 & $c$ & $b$ & $b$ & $a$ & $a$  \\
 \hline
   $\rho_4$ & $a$ & $b$ & $c$ & $c$ & $b$ & $b$ & $a$ & $a$  \\
 \hline
\end{tabular}
\vspace{3mm}
\caption{$V_4$-Representations}
\label{t1}
\end{table}

\begin{proof}[Proof of \autoref{p1.6}]
Let $\rho: \pi_1^{\orb}(X) \to SO(3)$ be one of the representations in \autoref{t1}. With the discussion above, one can verify easily that $\rho$  gives rise to an $SO(3)$-bundle $P$ with its unique irreducible flat connection $A$ up to $SO(3)$-gauge transformations. Let's write $\G'$ for the group of $SO(3)$-gauge transformations. The double cover $\eta: SU(2) \to SO(3)$ gives us an exact sequence (c.f. \cite[Page 66]{AMR95}): 
\[
1 \to \Map_{\orb}(X, \Z/2) \to \G \xrightarrow{\eta_*} \G' \to H^1_{\orb}(X; \Z/2) \to 1. 
\]
Since the stabilizer of $A$ in $\G'$ is $V_4$, and $\eta_*$ maps the center of $\G$ to $1$, we know the number of irreducible connections on $P$, up to determinant-1 gauge transformations, is 
\[
\rk (H^1(X; \Z/2)) / |V_4| = 2^6.
\]
For each $\G$-representative $A' \in \G'\cdot A$, \autoref{l6.3} tells us that the tangent space of the moduli space $\M_{\phi}(X, P^{\dagger})$ at $A'$ is $H^1(X, \ad_{A'})$ which is identified with $H^1(X, \ad_A) = 0$ canonically by any $SO(3)$-gauge transformations sending $A'$ to $A$. Thus the orientations of all the $2^6$ points agree, which tells us that $n_{\phi}(X, P^{\dagger}) = 2^6$. We note that this argument is also consistent with the computation in \autoref{ex6.1} where a $V_4$-representation gives rise to two projectively flat connections of the same sign up to determinant-$1$ gauge transformations.
\end{proof}

\end{exm}


\bibliographystyle{alpha}
\bibliography{RefG2}
\end{document}